\documentclass[letterpaper]{article}
\usepackage[intlimits]{amsmath}
\usepackage{amsfonts, amssymb, amsthm}
\usepackage{epsfig}
\usepackage{color}

\newlength{\hchng}
\newlength{\vchng}
\newtheorem{thm}{Theorem}[section]

\newtheorem{cor}[thm]{Corollary}
\newtheorem{lemma}[thm]{Lemma}
\newtheorem{defn}[thm]{Definition}

\newtheorem{assumption}[thm]{Assumption}

\newtheorem{preremark}[thm]{Remark}
\newenvironment{remark}{\begin{preremark}\rm}{\medskip \end{preremark}}
\numberwithin{equation}{section}

\newcommand{\norm}[1]{\left\Vert#1\right\Vert}
\newcommand{\abs}[1]{\left\vert#1\right\vert}

\newcommand{\R}{\mathbb R}
\newcommand{\eps}{\varepsilon}

\newcommand{\grad} {\nabla}
\newcommand{\lap} {\triangle}
\newcommand{\bdary} {\partial}
\newcommand{\dx} {\; \mathrm{d} x}
\newcommand{\dd} {\; \mathrm{d}}

\newcommand{\dist}{\mathrm{dist}}

\newcommand{\LI}{\mathcal{L}}
\newcommand{\MLp} {\mathrm{M}^+_\mathcal{L}}
\newcommand{\MLm} {\mathrm{M}^-_\mathcal{L}}
\newcommand{\si}{\delta}
\newcommand{\Mp} {\mathrm{M}^+}
\newcommand{\Mm} {\mathrm{M}^-}
\newcommand{\D}{\mathrm{I}}

\title{Regularity results for nonlocal equations by approximation}
\author{Luis Caffarelli and Luis Silvestre}

\begin{document}

\maketitle

\begin{abstract}
We obtain $C^{1,\alpha}$ regularity estimates for nonlocal elliptic equations that are not necessarily translation invariant using compactness and perturbative methods and our previous regularity results for the translation invariant case.
\end{abstract}

\section{Introduction}
We continue with the program started in \cite{CS} to extend the existing theory for (fully non linear) second order elliptic equations to nonlocal equations. In this paper we obtain nonlocal Cordes-Nirenberg type estimates. In \cite{CS} we obtained interior $C^{1,\alpha}$ estimates for nonlocal equations that are translation invariant. In this paper we extend those results by perturbative methods to non-translation invariant equations and other situations that were not covered in the results in \cite{CS}.

The main theorem in this paper (Theorem \ref{t:main}) is presented in a somewhat abstract setting in order to make it applicable to as many situations as possible. It is a general perturbative argument that says that if an equation is uniformly close to another one which has $C^{1,\alpha}$ solutions, then the solutions of the former are also $C^{1,\alpha}$. In the last section of the paper we present concrete applications. The equations we have in mind are the purely nonlocal Isaacs equations, which have the form
\begin{equation} \label{e:main}
\D(u,x) := \inf_{\alpha \in \mathcal{A}} \sup_{\beta \in \mathcal{B}} \int_{\R^n} (u(x+y)+u(x-y)-2u(x)) (2-\sigma) \frac{a_{\alpha \beta}(x,y)}{|y|^{n+\sigma}} \dd y = f(x) \qquad \text{in } \Omega.
\end{equation}
where $a_{\alpha \beta}$ is a family of nonnegative functions with indexes $\alpha$ and $\beta$ ranging in arbitrarily sets $\mathcal{A}$ and $\mathcal{B}$. This type of integro-differential equations appear in stochastic games with jump processes. For different assumptions on the kernels $a_{\alpha \beta}$ we obtain different regularity results. Some of the most interesting applications in the last section of this paper include the following.
\begin{itemize}
\item if $0 < \lambda \leq a_{\alpha \beta} \leq \Lambda$ and $|\grad_y a_{\alpha \beta}| \leq C |y|^{-1}$, and $a_{\alpha \beta}(x,y)$ is continuous in $x$ for a modulus of continuity independent of $y$, then a solution $u$ of \eqref{e:main} is $C^{1,\eps}$ for some $\eps>0$ in the interior of $\Omega$ (with an estimate depending on $\norm{u}_{L^\infty}$).
\item if $0 < \lambda \leq a_{\alpha \beta} \leq \Lambda$ and $|\grad_y a_{\alpha \beta}| \leq C |y|^{-1}$, $a_{\alpha \beta}$ is constant in $x$, and $\Lambda - \eta \leq 1 \leq \lambda + \eta$ for a small enough $\eta$, then the solution $u$ of \eqref{e:main} is $C^{2,\eps}$ for some $\eps>0$ in the interior of $\Omega$ (with an estimate depending on $\norm{u}_{L^\infty}$).
\end{itemize}

The equations we are most interested in arise from stochastic control problems and stochastic games (See \cite{CS}). However we develop a general theory that applies as soon as some abstract assumptions are satisfied by a nonlocal equation. Note that in the expression \eqref{e:main}, it only makes sense to consider symmetric kernels $a_{\alpha \beta}(x,y)= a_{\alpha \beta}(x,-y)$. Indeed, the antisymmetric part of the kernel would have no effect when we integrate it against the second order incremental quotients of $u$ (which are inherently symmetric).

In \cite{BarlesImbert} the regularity of viscosity solutions of nonlinear integro-differential equations that are not translation invariant is also studied. They obtain H\"older continuity of the solutions. Their hypothesis are quite different from ours. They require certain modulus of continuity of the $x$ dependence in the kernels. Their equations also involve second and first order terms, and they are able to include some degenerate equations into their hypothesis. In \cite{MR1432798} a parabolic and convex integro-differential equation is considered and it is proved that the solutions are Lipschitz in space and $C^{1/2}$ H\"older continuous in time under appropriate conditions.

In the translation invariant case, the $C^{1,\alpha}$ estimate is a somewhat general fact for equations satisfying a comparison principle due to the fact that the directional derivatives of solutions satisfy some (a priori discontinuous) elliptic equation. If we try to apply the same idea for $x$-dependent elliptic equations we usually need to impose an excessive amount of regularity on the dependence. An alternative is to use compactness properties of equations and solutions (see \cite{C2} and \cite{CC}). In the local (classical PDEs) case the approximation takes place in a bounded domain (say the unit ball) and simply consists in showing (by compactness) that if the oscillation of the dependence on $x$ of the variable equation is small enough, any bounded solution is close to the solution to the translation invariant equation inheriting its regularity up to a small error. 

A good portion of the theory depends very little on the actual integral form of the operator $\D(u,x)$ and it gives rise to interesting applications. Therefore, more generally, by a nonlocal operator $\D$ we will mean a rule that assigns a value $\D(u,x)$ for any function $u$ that is $C^2(x)$ and does not grow too much at infinity. A safe assumption would be that $u$ is bounded since the kernels decay at infinity, but for particular cases it is enough to assume that $u \in L^1(\R^n,\omega)$ for a suitable weight $\omega$ related to the size of the nonlocal interactions (more concretely, a bound for the tails in the kernels $K_{\alpha \beta}$). We also assume that $\D(u,x)$ is continuous if $u \in C^2$. Degenerate ellipticity is understood in the following way, if $v \geq u$ and $v(x) = u(x)$ for some $x$ then $\D(v,x) \geq \D(u,x)$. Note that the usual elliptic second order operators can be thought of as a special case of this nonlocal elliptic operators, and also as a limiting case of integral operators of the form written above.

For every fixed $x_0$, the operator $\D_{x_0} = \D( - , x_0)$ given by $\D_{x_0} (u,x) = \D(\tau_{x-x_0} u,x_0)$, where $\tau_x u(y) = u(x+y)$, is translation invariant. This is just notation, which corresponds to \emph{freezing} the coefficients in the classical linear theory.

As mentioned above, we need compactness of the family of solutions to equations with dependence on $x$, like De Giorgi or Krylov-Safonov theory for local operators. This we are able to do in the case of \emph{uniform ellipticity} with respect to a family of linear operators. We say that $\D$ is uniformly elliptic with respect to a family of linear operators $\LI$ if
\begin{equation} \label{e:generaluniformellipticity}
 \MLm v (x) \leq \D(u+v,x) - \D(u,x) \leq \MLp v(x)
\end{equation}
where 
\begin{align*}
\MLm v (x) &:= \inf_{L \in \LI} Lv(x) \\
\MLp v (x) &:= \sup_{L \in \LI} Lv(x)
\end{align*}

This is the same as the definition in \cite{CS}, but now we do not take the operator $\D$ necessarily translation invariant, and we want to include the case that $\D$ may be a second order partial differential operator (local operator). For second order equations, this is in fact the definition in \cite{C2} with respect to the Pucci extremal operators.

The class $\LI$ represents the set of kernels that are allowed in an Isaacs equation like \eqref{e:main}. Depending on the type of regularity results that we want to obtain, we need to impose less or more conditions on these kernels. That is the reason why we need to keep a flexible definition of ellipticity with allows us to consider different classes $\LI$. A common class of operators is the class $\LI_0$ of uniformly elliptic operators of order $\sigma$ defined in \cite{CS} and also later in this paper. These are the linear operators whose kernels are comparable to that of the fractional Laplacian of order $\sigma$ but no regularity is imposed respect to $y$. It is the class for which H\"older estimates and the Harnack inequality is proved in \cite{CS}. However, in order to prove $C^{1,\alpha}$ estimates for a nonlinear problem like \eqref{e:main}, a more restrictive class is needed which imposes some regularity of the kernels respect to $y$.

For the class $\LI_0$, or any subclass of it, we can allow the functions $u$ to have certain growth at infinity. The exact condition would be that $\D(u,x)$ is well defined as long as $u \in C^{1,1}(x)$ and $u \in L^1(\R^n,\frac{1}{1+|y|^{n+\sigma}} )$ meaning that \[ \int_{\R^n} |u(y)| \frac{1}{1+|y|^{n+\sigma}} \dd y < +\infty . \]
Clearly, assuming that $u$ is bounded is a sufficient condition but it is not necessary. In particular, for $\sigma>1$, we can still compute $\D(u,x)$ if $u$ has linear growth at infinity.

The \textbf{main result} of this paper is that if some fixed equation $\D^{(0)} u(x) = 0$ has interior $C^{1,\bar \alpha}$ estimates, and $\D(-,x)$ stays close to $\D^{(0)}$ in every scale as $x$ varies in a given domain, then the equation $\D(u,x)=f(x)$ has interior $C^{1,\alpha}$ estimates for some $\alpha < \bar \alpha$. This is a nonlinear fractional-order version of the classical Cordes-Nirenberg estimate for second order elliptic equations.

For linear elliptic PDEs, the classical Cordes-Nirenberg estimate (\cite{Cordes}, \cite{Niremberg}) says that an equation
\[ a_{ij}(x) \partial_{ij} u(x) = f(x) \]
has interior $C^{1,\alpha}$ estimates, provided that $f \in L^\infty$ and $|a_{ij}(x) - \delta_{ij}| \leq \eps$ for a small enough $\eps$. In \cite{C2}, this regularity result was generalized to fully nonlinear equations. Our results in this paper recover the second order case since our estimates are uniform with respect to the order of the equation.

The main theorem in this paper is given in section \ref{s:estimate}. The proof follows more or less the idea of the proof in \cite{CC} chapter 8, but taking care of the nonlocal effects of the values of the function away from the domain of the equation. It is an approximation type argument, based on Lemma \ref{l:approximation} in section \ref{s:approximationresults} that says essentially that if two equations are very close, then so are their solutions. 

It is necessary to give a meaning for two nonlocal equations to be close that will not impose restrictive conditions in the long range interactions and will still provide the necessary compactness. For this purpose, a suitable norm is defined in a preliminary way in section \ref{s:preliminaries} for nonlocal operators. In section \ref{s:estimate}, the definition of the norm is refined so that it represents that two operators are close even when we rescale them. 

The proof of Lemma \ref{l:approximation} is by compactness, so it requires that the equations that we deal with have continuous solutions up to the boundary. This is why we have section \ref{s:boundaryregularity} where we construct appropriate barriers to show that solutions to uniformly elliptic nonlocal equations are continuous up to the boundary.

Our main theorem is presented in a very general form in section \ref{s:estimate}. So in section \ref{s:applications} we provide a list of more concrete applications. The study of the regularity properties of nonlinear integro-differential equations is a relatively new subject. This paper is part of a program trying to extend the fundamental regularity results of fully nonlinear elliptic PDEs to nonlocal equations. The family of integro-differential equations is very rich in the sense that we can come up with a very large variety of equations. We try to provide theorems that apply to a general class of equations. There is always a balance between generality of results and simplicity of exposition and at this moment the choice is not always easy to make.

\section{Preliminaries: Convergence of operators, uniform ellipticity and viscosity solutions.} \label{s:preliminaries}
In this section we introduce and discuss closeness of operators, uniform ellipticity and the notion of viscosity solutions. We also recall the regularity results of \cite{CS} and \cite{S1} for uniformly elliptic integro-differential equations. We also recall the interior $C^{1,\alpha}$ estimates from \cite{CS} with the hypothesis slightly relaxed.

We start with a definition of nonlocal operators (Definition \ref{d:operators}). We next describe in what sense we will say that two operators are close. Since our approach is based on the viscosity method, closeness is measured by testing against smooth functions (Definition \ref{d:norm}). Next we define and discuss in what context our operators are uniformly elliptic. At that point we are ready to define what a viscosity solution to our operators is (Definition \ref{d:viscositysolutions}). Then we review the regularity theorems proved in \cite{CS} and \cite{S1} about viscosity solutions.

In \cite{S1}, there is a very simple proof that solutions of integro-differential equations with variable order are H\"older continuous, as long as the order of the equation stays strictly away from $0$ and $2$. The more complicated regularity estimates in \cite{CS} hold even when the order of the equation gets arbitrarily close to $2$, however, only equations with constant order are considered. To a great extent, this was a choice for simplicity in the presentation, since many of the results in \cite{CS} would hold for equations of variable order, but the proofs would be somewhat messier. The results we present in this section are only a small variation of the results in \cite{CS}. For our main theorem in section \ref{s:estimate}, we will make a perturbation argument around an equation with interior $C^{1,\alpha}$ estimates. The applications in section \ref{s:applications} would extend with any extension of our theorems \ref{t:ca} and \ref{t:c1a} below.

We will consider an absolutely continuous weight $\omega$ depending on the type of operators that we want to deal with. The purpose of this weight is to measure the contribution of the tails of the integral in the computation of an integro-differential operator $\D u$.  Typically we have in mind $\omega(y) = 1/(1+|y|^{n+\sigma_0})$ for some $\sigma_0 \in (1,2)$, so that $u \in L^1(\R^n,\omega)$ means that the integrals in \eqref{e:main} are not singular at infinity for any $\sigma \geq \sigma_0$. Throughout this paper we always make the following assumptions on $\omega$.
\begin{align}
1+|y| &\in L^1(\R^n,\omega) \label{e:omegalinear} \\
\sup_{B_r(y)} w &\leq C(r) w(y) \label{e:omegadisks}
\end{align}

These are the only things we need to assume on the weight $\omega$ for the general theorems of section \ref{s:approximationresults}. 
The results in this paper do not apply if the order $\sigma$ of the equation is smaller than one, in that respect assumption \eqref{e:omegalinear} is reasonable. Assumption \eqref{e:omegadisks} is more technical and rules out singular measures $\omega$. 

We use the following convenient definition for a nonlocal operator.

\begin{defn} \label{d:operators}
A nonlocal operator $I$ is a rule that assigns to a function $u$ a value $\D(u,x)$ at every point $x$, satisfying the following assumptions.
\begin{itemize}
\item $\D(u,x)$ is well defined as long as $u \in C^2(x)$ and $u \in L^1(\R^n,\omega)$.
\item If $u \in C^2(\Omega) \cap L^1(\R^n,\omega)$, then $I(u,x)$ is continuous in $\Omega$ as a function of $x$.
\end{itemize}
\end{defn}

By $u \in C^2(x)$ we mean that there is a quadratic polynomial $q$ such that $u(y) = q(y) + o(|x-y|^2)$ for $y$ close to $x$.

Definition \ref{d:operators} is very general. It becomes useful when complemented with the concept of (nonlocal) uniform ellipticity \eqref{e:generaluniformellipticity}. We do not include \eqref{e:generaluniformellipticity} in the definition above because we want to define the norm of a general operator (defined below) and compute it for differences of elliptic operators.

The results in this paper are based on approximation techniques. In order to use that type of argument, we need to understand when two nonlocal operators are close to each other. Since we use the Crandall-Lions viscosity concept, it is natural to define closeness of operators by testing them against the appropriate family of functions. We define the following norm.

\begin{defn} \label{d:norm}
Given a nonlocal operator $\D$, we define $\norm{\D}$ in some domain $\Omega$ with respect to some weight $\omega$ as
\[ \begin{aligned} \norm{\D} := \sup \big\{ |\D(u,x)| / (1+M) : \ \ &x \in \Omega \\ 
                                         & u \in C^2(x) \\ 
                                         &\norm{u}_{L^1(\R^n,\omega)} \leq M \\ 
                                         &|u(y)-u(x)-(y-x)\cdot \grad u(x)| \leq M |x-y|^2 \text{ for any } y \in B_1(x)  \big\}.
                                        \end{aligned} \]
\end{defn}

The idea of this definition of $\norm{\D}$ is that $\norm{\D_k-\D} \to 0$ if $\D_k(u,x) \to \D(u,x)$ in a somewhat uniform way. The condition $\norm{u}_{L^1(\R^n,\omega)} \leq M$ could be replaced by $|u| \leq M \text{ in } \R^n$ which would make convergence weaker.

Note that in the \emph{local} case when $\D (u,x) = F(D^2 u,x)$, we have $\omega=0$ and $\norm{\D} = \tilde \beta(F)$, where $\tilde \beta$ is the one from \cite{CC}. 

Our uniform ellipticity assumption \eqref{e:generaluniformellipticity} depends on the class $\LI$ of integro-differential operators. A class $\LI$ is a set of operators $L$ of the form
\[ L u(x) = \int_{\R^n} (u(x+y) + u(x-y) - 2u(x)) K(y) \dd y, \]
for a family of nonnegative kernels $K$ satisfying $\int \min(1,|y|^2) K(y) \dd y < +\infty$.


We will make some minimal assumptions on every class $\LI$. In this whole paper, every class $\LI$ will satisfy the following two assumptions, that connect the class with the choice of the weight $\omega$.
\begin{assumption} \label{a:kernelK}
If $K = \sup K_\alpha$ is the supremum of all kernels in the class, then for every $r>0$
\[K(y) \leq C_r \omega(y) \qquad \text{if $|y| \geq r$.} \]
for some constant $C_r$ (depending on $r$).
\end{assumption}

\begin{assumption} \label{a:boundedness}
For some constant $C$ and every $L \in \LI$, we have $\norm{L} \leq C$
(which clearly implies $\norm{\MLp} \leq C$ and also $\norm{\MLm} \leq C$).
\end{assumption}

Assumption \ref{a:kernelK} says that the weight $\omega$ controls the tails of the kernels at every scale. Assumption \ref{a:boundedness} says that near the origin the kernels are uniformly controlled by $C^2$ functions. These are very weak assumptions and are all that is needed for our stability results of section \ref{s:approximationresults}. Our regularity results, however, require stronger assumptions in the class $\LI$.

In order to understand the assumptions better, let us consider $\LI$ to be the set of all fractional Laplacians $-(-\lap)^{\sigma/2}$ for $\sigma \in [\sigma_0,2]$, it would satisfy both assumptions with $\omega(y) = (2-\sigma_0)/(1+|y|^{n+\sigma_0})$. On the other hand if we take the set $\LI$ of all operators with kernels of the form $1/|y|^{n+\sigma}$ (where we are omitting the normalizing factor $\approx (2-\sigma)$), then this class would \textbf{not} satisfy Assumption \ref{a:boundedness}.

In the case of local equations, the class $\LI$ would consist of linear second order partial differential operators, and the weight $\omega$ is irrelevant (it can be taken as $\omega=0$ and also $K=0$). Our results in this case reduce to the usual results for uniformly elliptic partial differential equations \cite{CC}.

We recall the definition of viscosity solutions given in \cite{CS}.

\begin{defn}[Viscosity Solutions] \label{d:viscositysolutions}
Let $I$ be a uniformly elliptic operator in the sense of \eqref{e:generaluniformellipticity} with respect to some class $\LI$. A function $u :\R^n \to \R$, upper (lower) semi continuous in $\overline \Omega$, is said to be a subsolution (supersolution) to $\D(u,x) = f(x)$, and we write $\D(u,x)\geq f(x)$ ($\D(u,x) \leq f(x)$) in $\Omega$, if every time all the following happens
\begin{itemize}
\item $x_0$ is any point in $\Omega$.
\item $N$ is a neighborhood of $x_0$ in $\Omega$.
\item $\varphi$ is some $C^2$ function in $\overline N$.
\item $\varphi(x_0) = u(x_0)$.
\item $\varphi(x) > u(x)$ ($\varphi(x) < u(x)$) for every $x \in N \setminus \{x_0\}$.
\end{itemize}
Then if we choose as test function $v$, \[ v := \begin{cases}
               \varphi &\text{in } N \\
	       u &\text{in } \R^n \setminus N \ ,
              \end{cases} \]
we have $\D(v,x_0) \geq f(x_0)$ ($\D(v,x_0) \leq f(x_0)$).

A solution is a function $u$ which is both a subsolution and a supersolution.
\end{defn}

The first theorem we borrow from \cite{CS} is the interior H\"older estimates. We change the assumption of $u$ being bounded for $u \in L^1(\R^n,1/(1+|y|^{n+\sigma_0}))$. Let us first recall the class $\LI_0$ of integro differential operators.

An integro differential operator $L$ belongs to $\LI_0$ if its kernel $K$ satisfies 
\begin{equation} \label{e:uniformellipticity} 
 (2 - \sigma) \frac{\lambda}{|y|^{n+\sigma}}\leq K(y) \leq (2 - \sigma) \frac{\Lambda}{|y|^{n+\sigma}} \, .
\end{equation}
In this case $\Mp_{\LI_0}$ and $\Mm_{\LI_0}$ take a very simple form: 
\begin{align}
\Mp_{\LI_0} v(x) &= (2 - \sigma) \int_{\R^n} \frac{\Lambda \si v(x,y)^+ - \lambda \si v(x,y)^-}{|y|^{n+\sigma}} \dd y \label{e:Mp}\\
\Mm_{\LI_0} v(x) &= (2 - \sigma) \int_{\R^n} \frac{\lambda \si v(x,y)^+ - \Lambda \si v(x,y)^-}{|y|^{n+\sigma}} \dd y \ . \label{e:Mm} 
\end{align}
where $\si$ represents the second order incremental quotients $\si u(x,y) = u(x+y) + u(x-y) - 2u(x)$.

The class $\LI_0$ depends on the parameter $\sigma \in (0,2)$. Whenever it is necessary, we will stress that dependence by writing $\LI_0(\sigma)$.

The next theorem concerns equations with kernels that can be arbitrarily discontinuous in terms of $x$. In particular, it applies to the difference of two solutions of a nonlinear integral operator.

\begin{thm} \label{t:ca}
Let $\sigma>\sigma_0$ for some $\sigma_0>0$. Let $u$ be function in $L^1(\R^n,\omega)$ for $\omega=1/(1+|y|^{n+\sigma_0})$, such that $u$ is continuous in $\overline B_1$ and 
\begin{align*}
\Mp_{\LI_0} u &\geq -C_0 \qquad \text{in } B_1 \\
\Mm_{\LI_0} u &\leq C_0 \qquad \text{in } B_1
\end{align*}
then there is an $\alpha > 0$ (depending only on $\lambda$, $\Lambda$, $n$ and $\sigma_0$) such that $u \in C^\alpha(B_{1/2})$ and
\[ u_{C^\alpha(B_{1/2})} \leq C \big( \sup_{B_1} |u| + \norm{u}_{L^1(\R^n,\omega)} + C_0 \big) \]
for some constant $C>0$ which depends on $\sigma_0$, $\lambda$, $\Lambda$ and dimension, but not on $\sigma$.
\end{thm}

\begin{proof}
Consider $\tilde u(x) = \chi_{B_1}(x) u(x)$. A direct computation shows that $\Mp \tilde u \geq -C_0 - c \norm{u}_{L^1(\R^n,\omega)}$ and $\Mm \tilde u \leq C_0 + c \norm{u}_{L^1(\R^n,\omega)}$ in $B_{3/4}$. We can now apply theorem 12.1 from \cite{CS} to $\tilde u$. 
\end{proof}

Theorem \ref{t:ca} is a nonlocal version of Krylov-Safonov theorem for elliptic PDEs in nondivergence form. A nonlocal version of De Giorgi-Nash-Moser theorem can be found in \cite{kassmann2009priori}.


We will also mention, as a consequence of the previous theorem, the interior $C^{1,\alpha}$ estimates from \cite{CS} for functions with some growth at infinity.

Given $\rho_0>0$, we define the class $\LI_\ast$ by the operators $L$ with kernels $K$ such that
\begin{align} 
(2-\sigma)\frac{\lambda}{|y|^{n+\sigma}} &\leq K(y) \leq (2-\sigma)\frac{\Lambda}{|y|^{n+\sigma}} \label{e:lic1a1}\\ 
|\grad K(y)| &\leq C \omega(y) \qquad \text{in } \R^n \setminus B_{\rho_0} \label{e:lic1a2}
\end{align}

\begin{thm} \label{t:c1a} Assume $\sigma \geq \sigma_0>0$. There is a $\rho_0>0$ (depending on $\lambda$, $\Lambda$, $\sigma_0$ and $n$) so that if $\D$ is a nonlocal translation invariant uniformly elliptic operator with respect to $\LI_\ast$ and $u$ is a continuous function in $\overline B_1$ such that $u \in L^1(\R^n,\omega)$ for $\omega=1/(1+|y|^{n+\sigma_0})$, $\D u = 0$ in $B_1$,
then there is a universal (depends only on $\lambda$, $\Lambda$, $n$ and $\sigma_0$) $\alpha>0$ such that $u \in C^{1+\alpha}(B_{1/2})$ and
\[ u_{C^{1+\alpha}(B_{1/2})} \leq C\left( \sup_{B_1} |u| + \norm{u}_{L^1(\R^n,\omega)} +|\D 0| \right) \]
for some constant $C>0$ (where by $\D 0$ we mean the value we obtain when we apply $\D$ to the function that is constant equal to zero). The constant $C$ depends on $\lambda$, $\Lambda$, $\sigma_0$, $n$ and the constant in \eqref{e:lic1a2}.
\end{thm}

\begin{proof}[Sketch of the proof.]
The proof of this theorem goes along the lines of the proof of Theorem 12.1 in \cite{CS} by applying Theorem \ref{t:ca} to the differential quotients $w^h = (u(x+h)-u(x))/|h|^\beta$ for $\beta = \alpha, 2\alpha, \dots, 1$. In each step, the quotient $w^h$ is written as $w^h = w^h_1 + w^h_2$ where $w^h_2$ vanishes in a ball $B_r$ (with $1/2 < r <1$) and $w_1^h$ is bounded in $\R^n$, $\Mm w^h_1 \leq C$ and $\Mp w^h_1 \geq -C$ in $B_r$. This constant $C$, like in the proof of theorem 12.1 in \cite{CS}, comes from the estimate
\[ \int_{\R^n \setminus B_\rho} |u(x+y)| \frac{ |K(y) - K(y-h)| }{|h|} \dd y  < C .\]
In \cite{CS} the assumption was $u \in L^\infty$ whereas now we have $u \in L^1(\R^n,\omega)$ and so is why the assumption (12.2) in \cite{CS} necessary for the bound of the constant $C$ is different from our assumption \eqref{e:lic1a2} in this paper.
\end{proof}

Note that $\sup_{B_1} |u| + \norm{u}_{L^1(\R^n,\omega)} \leq C \sup |u(x)|/(1+|x|^{1+\alpha})$ for any $\alpha < \sigma_0$.

In comparison with $\LI_0$, the class $\LI_\ast$ adds the extra regularity assumption \eqref{e:lic1a2} to the kernels with respect to $y$. This condition is needed in order to prove the interior $C^{1,\alpha}$ estimates in the ball of radius one. However, the hypothesis is not scale invariant, which makes it not suitable to the type perturbation arguments used later in this paper. Therefore, in section \ref{s:estimate} we will further restrict this class in order to make it scale invariant.

\begin{remark}
Assumption \eqref{e:lic1a2} is used to control the effect on the regularity of $u$ at $x$ of the points $y$ that are at least $\rho_0$ away from $x$. This condition depends on the type of control we have of $u$ far from $x$. In \cite{CS} we were assuming the stronger condition $u \in L^\infty(\R^n)$, so we had the weaker condition for the kernels
\[\int_{\R^n \setminus B_{\rho_0}} \frac{|K(y)-K(y-h)|}{|h|} \dd y \leq C \qquad \text{every time $|h|<\frac {\rho_0} 2$}\]
Now we have the weaker condition $u \in L^1(\R^n,\omega)$ instead. If we want our estimate to depend on different norms of $u$, our condition \eqref{e:lic1a2} will change accordingly.
\end{remark}

It is also convenient to define $M^+_2$ and $M^-_2$ as
\begin{align*}
M^+_2 u(x) = \lim_{\sigma \to 2} M^+_{\LI_0(\sigma)} u(x) \\
M^-_2 u(x) = \lim_{\sigma \to 2} M^-_{\LI_0(\sigma)} u(x)
\end{align*}

A simple computation shows that $M^+_2 u(x)$ is a second order uniformly elliptic operator of the form $M^+_2 u(x) = F(D^2 u)$ with
\[ F(A) := \int_{S^{n-1}} \Lambda \langle y , A y \rangle^+ - \lambda \langle y , A y \rangle^- \dd S. \] 
Note that the ellipticity constants $\tilde \lambda$ and $\tilde \Lambda$ of $M^+_2$ as a partial differential operator depend only on $\lambda$, $\Lambda$ and dimension. Also $F(0)=0$, so $M^+_2 u(x) \leq M^+(D^2 u)$ where $M^+(D^2 u)$ is the usual Pucci operator with ellipticity constants $\tilde \lambda$ and $\tilde \Lambda$. Naturally, we also have the corresponding relations for $M^-_2$.

\section{Boundary regularity} \label{s:boundaryregularity}
In this section we show that if we have a modulus of continuity on the boundary of the domain of some equation, then we can find a (possibly different) modulus of continuity inside the domain. This is done, as usual, by controlling the growth of $u$ away from its boundary data through barriers, scaling and interior regularity. 

We start by showing that a truncation of an appropriate radial function is a barrier outside the unit ball. Our barrier function is good as a supersolution $\Mp \varphi \leq 0$ for all values of $\sigma$ greater than a given $\sigma_0$, where $\Mp$ is the maximal operator $\Mp_{\LI_0(\sigma)}$. Another way to say this would be to define a larger class $\LI$ which is the union of all classes $\LI_0(\sigma)$ for $\sigma \in (\sigma_0,2)$, then $\MLp \varphi \leq 0$.

Using this barrier function we will prove that the solutions of uniformly elliptic equations are continuous on the boundary, and then using their interior H\"older regularity (Theorem \ref{t:ca}) we get that they are continuous up to the boundary in the domain of the equation. The result in Theorem \ref{t:boundarycontinuity} is restricted to equations of constant order only because Theorem \ref{t:ca} is. In fact, the same continuity result holds for operators of variable order as long as the order remains strictly in between $0$ and $2$ (because of the results in \cite{S1}).

\begin{lemma} \label{l:barrier}
Given any $\sigma_0 \in (0,2)$, there is an $\alpha>0$ and $r>0$ small so that the function $u(x) = ((|x|-1)^+)^\alpha$ satisfies $M^+ u(x) \leq 0$ in $B_{1+r} \setminus B_1$ for any $\sigma>\sigma_0$.
\end{lemma}

We defer the proof of this lemma to the appendix. We use it to obtain the following corollary.

\begin{cor} \label{c:barrier}
For any constant $C$, there is a continuous function $\varphi$ such that
\begin{itemize}
\item $\varphi = 0$ in $B_1$.
\item $\varphi \geq 0$ in $\R^n$.
\item $\varphi \geq 1$ in $\R^n \setminus B_2$.
\item $\Mp \varphi \leq 0$ in $\R^n \setminus B_1$ for any $\sigma>\sigma_0$.
\end{itemize}
\end{cor}

\begin{proof}
Consider $\varphi = \min(1, C ((|x|-1)^+)^\alpha)$ for some large constant $C$ and apply Lemma \ref{l:barrier}.
\end{proof}

We will use the function $\varphi$ of Corollary \ref{c:barrier} as a barrier to prove the boundary continuity of solutions to nonlocal equations. We will do it in spherical domains. Note that $\varphi$ is a supersolution outside of a ball, so the same method could be used in domains with the exterior ball condition.

\begin{thm} \label{t:boundarycontinuity}
Let $\sigma > \sigma_0 >0$. Let $\rho$ be a modulus of continuity and $\Mp$ and $\Mm$ be the maximal operators $\Mp_{\LI_0(\sigma)}$ and $\Mm_{\LI_0(\sigma)}$. Let $u$ be a bounded function such that
\begin{align*}
\Mp u (x) &\geq -C && \text{in } B_1\\
\Mm u (x) &\leq C && \text{in } B_1\\
|u(y) - u(x)| &\leq \rho(|x-y|) && \text{for every $x \in \bdary B_1$ and $y \in \R^n \setminus B_1$.}
\end{align*}
Then there is another modulus of continuity $\tilde \rho$ so that $|u(y) - u(x)| \leq \tilde \rho(|x-y|)$ for every $x \in \overline B_1$ and $y \in \R^n$. 

The modulus of continuity $\tilde \rho$ depends only on $\rho$, $\lambda$, $\Lambda$, $\sigma_0$, $n$, $\norm{u}_{L^\infty}$ and $C$ (but not on $\sigma$).
\end{thm}

The modulus of continuity $\tilde \rho$ in Theorem \ref{t:boundarycontinuity} is in some sense the worst between the internal regularity and the boundary regularity.

Before proving Theorem \ref{t:boundarycontinuity}, we mention a corollary.

\begin{cor} \label{c:boundarycontinuity}
Let $\sigma > \sigma_0 >0$. Let $\rho$ be a modulus of continuity. Let $u$ be a function such that
\begin{align*}
\norm{u}_{L^1(R^n,1/(1+|y|^{n+\sigma_0}))} &\leq C \\
\Mp u (x) &\geq -C && \text{in } B_1\\
\Mm u (x) &\leq C && \text{in } B_1\\
|u(y) - u(x)| &\leq \rho(|x-y|) && \text{ for every $x \in \bdary B_1$ and $y \in \R^n \setminus B_1$.}
\end{align*}
Then there is another modulus of continuity $\tilde \rho$ so that $|u(y) - u(x)| \leq \tilde \rho(|x-y|)$ for every $x \in \overline B_1$ and $y \in \R^n$. 

The modulus of continuity $\tilde \rho$ depends only on $\rho$, $\lambda$, $\Lambda$, $\sigma_0$, $n$, and $C$.
\end{cor}

\begin{proof}[Proof of corollary]
Since $|u(y)-u(x)|\leq \rho(|x-y|)$ if $x \in \bdary B_1$ and $\int_{B_2} |u| \dx < C$, then $|u|\leq C$ on $\bdary B_1$. Thus, from the modulus of continuity, $u \leq C$ in $B_2$.

We apply Theorem \ref{t:boundarycontinuity} to $\tilde u(x) = \min(C,\max(u,C))$. Note that $\Mp \tilde u \leq \Mp u + \Mp (\tilde u - u)$ and $\Mp (\tilde u - u)$ is controlled by $\norm{u}_{L^1(R^n,1/(1+|y|^{n+\sigma_0}))}$.
\end{proof}

In order to prove Theorem \ref{t:boundarycontinuity}, we will need a couple of preparatory lemmas that mirror the second order approach. Lemma \ref{l:bcont} reconstructs the modulus of continuity on the boundary, and Lemma \ref{l:insidecont} in the interior.

\begin{lemma} \label{l:bcont}
Assume $\sigma \geq \sigma_0 > 0$. Let $\rho$ be a modulus of continuity. Let $u$ be a bounded function such that
\begin{align*}
\Mp u (x) &\geq -C \\
u(y) - u(x) &\leq \rho(|x-y|) && \text{ for every $x \in \bdary B_1$ and $y \in \R^n \setminus B_1$.}
\end{align*}
Then there is another modulus of continuity $\tilde \rho$ so that $u(y) - u(x) \leq \tilde \rho(|x-y|)$ for every $x \in \bdary B_1$ and $y \in \R^n$. 

The modulus of continuity $\tilde \rho$ depends only on $\rho$, $\lambda$, $\Lambda$, $\sigma_0$, $n$, the constant $C$ above, and $\norm{u}_{L^\infty}$.
\end{lemma}

\begin{proof}
We first observe that since $\sigma \geq \sigma_0$ the function $p(x) = \max(0,4-|x|^2)$ satisfies $M^+ p \geq -c$ in $B_1$ for some $c>0$. So we can consider $u-\frac{C}{c} p$ and reduce the problem to the case $C=0$ (the values of $\norm{u}_{L^\infty}$ and $\rho$ are affected depending on $C$).

Let $\rho_0$ be the modulus of continuity of the function $\varphi$ of Corollary \ref{c:barrier}. Let $\tilde \rho$ be
\[ \tilde \rho(r) = \inf_{R} \left( \rho(3R) + \norm{u}_{L^\infty} \rho_0\left(\frac{r}{R}\right) \right) \] 

Let us first show that $\tilde \rho$ is indeed a modulus of continuity. The function $\tilde \rho$ is clearly monotone increasing because $\rho_0$ is. We still need to show that for every $\eps>0$, there is an $r>0$ such that $\tilde \rho(r) < \eps$. We first choose $R$ such that $\rho(3R) < \eps/2$, and then choose $r$ such that $\norm{u}_{L^\infty} \rho_0\left(\frac{r}{R}\right) < \eps /2$.

Now, we show that $\tilde \rho(|y-x|) \geq u(y) - u(x)$ for any $y \in \R^n$ and $x \in \bdary B_1$. For any $R>0$, Consider the barrier function
\[ B(y) = u(x) + \rho(3R) + \norm{u}_{L^\infty} \varphi\left(-x + \frac{y-x}{R}\right) . \]

We see that $B(x) \geq u(x) + \rho(3R) \geq u$ in $B_{3R}(x) \cap (\R^n \setminus B_1)$. Also $B(x) \geq \norm{u}_{L^\infty}$ in $\R^n \setminus B_{3R}(x)$. Moreover $M^+ B \leq 0$, whereas $M^+ u \geq 0$ in $B_1$, so $u \leq B \leq u(x) + \tilde \rho(|y-x|)$ everywhere. Taking infimum from all possible choices of $R$ we obtain $\tilde \rho(|y-x|) \geq u(y) - u(x)$.
\end{proof}

\begin{lemma} \label{l:insidecont}
Assume $\sigma > \sigma_0 >0$. Let $\rho$ be a modulus of continuity. Let $u$ be a bounded function such that
\begin{align}
\Mp u (x) &\geq -C && \text{ in } B_1 \label{e:formp}\\
\Mm u (x) &\leq C && \text{ in } B_1 \\
|u(y) - u(x)| &\leq \rho(|x-y|) && \text{ for every $x \in \bdary B_1$ and $y \in \R^n$.}
\end{align}
Then there is another modulus of continuity $\tilde \rho$ so that $|u(y) - u(x)| \leq \tilde \rho(|x-y|)$ for every $x \in \overline B_1$ and $y \in \R^n$. 

The modulus of continuity $\tilde \rho$ depends only on $\rho$, $\lambda$, $\Lambda$, the constant $C$ above, $\sigma_0$, $n$, and $\norm{u}_{L^\infty}$.
\end{lemma}

\begin{proof}
Let $x \in \overline B_1$ and $y \in \R^n$. We will consider two cases. If $\dist(x,y) < \dist(x,\bdary B_1)/4$ or if $\dist(x,y) \geq \dist(x,\bdary B_1)/4$.

If $\dist(x,y) \geq \dist(x,\bdary B_1)/4$ then let $x_0$ be a point in $\bdary B_1$ where $\dist(x,x_0) = \dist(x,\bdary B_1)$, 
\begin{align*}
|u(x) - u(y)| &\leq |u(x)-u(x_0)| + |u(x_0) - u(y)| \\
&\leq \rho(|x-x_0|) + \rho(|x_0-y|) \\ &\leq \rho(|x-y|/4) + \rho(5|x-y|/4)
\end{align*}

Let us consider now the case $\dist(x,y) < \dist(x,\bdary B_1)/4$. Let $r = \dist(x,\bdary B_1)/2$. Let us truncate the function $u$ at $u(x_0) \pm \rho(4r)$ by defining $\overline u := \min( u(x_0)+\rho(4r), \max(u,u(x_0)-\rho(4r)))$. So we have
\[ |u(z) - \overline u(z)| \leq \min( (\rho(r+|z-x|) - \rho(4r))^+ , M) \]
where $M = 2\norm{u}_{L^\infty}$.

We now scale the functions so that $B_r(x)$ becomes $B_1$. We write
\begin{align*}
\overline v(z) &= \overline u(x+rz) \\
v(z) &= u(x+rz)
\end{align*}
and we have
\[ |v(z) - \overline v(z)| \leq \min( (\rho(r+r|z|) - \rho(4r))^+ , M) . \]

We want to estimate $M^+ \overline v$ in $B_1$. We know that $M^+ v \geq -C r^\sigma$ in $B_1$ by rescaling \eqref{e:formp}. Also $M^+ \overline v \geq M^+ v - M^+ (v - \overline v)$. We now estimate $M^+(v - \overline v)$ in $B_1$ by above. We have
\begin{align*}
M^+(v(z) - \overline v(z)) &\leq (2-\sigma) \Lambda \int_{\R^n \setminus B_1} \frac{2\min( (\rho(r+r|y|) - \rho(4r))^+ , M)}{|y|^{n+\sigma}} \dd y \\
&=: I_r
\end{align*}
We observe that $I_r \to 0$ as $r \to 0$ uniformly in $\sigma$ as long as $\sigma \geq \sigma_0$. We thus have 
$-Cr^\sigma - I_r \leq M^+ \overline v$ in $B_1$, and $Cr^\sigma + I_r \to 0$ as $r \to 0$.

We can repeat the same computation to obtain $M^- \overline v \leq Cr^\sigma + I_r$ in $B_1$. Applying interior $C^\alpha$ estimates \cite{CS}, we obtain
\[
\overline v(z) - \overline v(0) \leq C (Cr^\sigma + I_r + \rho(4r)) |z|^\alpha = C m_r |z|^\alpha
\]
where $m_r := Cr^\sigma + I_r + \rho(4r)$ goes to zero as $r \to 0$ uniformly for $\sigma \geq \sigma_0$.

Now we scale back with $y=x+rz$ to apply the above estimate to $u$. We have
\[ |u(y) - u(x)| \leq C m_r \left( \frac{|x-y|}{r} \right)^\alpha \] 
so we have proven that $|u(y) - u(x)| \leq \tilde \rho(|x-y|)$ as long as we can chose $\tilde \rho(d) \geq C \sup_{r>4d}  \frac{d^\alpha m_r}{r^\alpha}$. There will be such modulus of continuity as long as we can prove that 
\begin{equation} \label{e:mccool}
\sup_{r>4d} \frac{d^\alpha m_r}{r^\alpha} \to 0 \text{  as  } d \to 0 
\end{equation}

We will prove this by contradiction to finish the proof of this lemma. Assume \eqref{e:mccool} was not true. Then there would be an $\eps>0$ and two sequences $r_i$ and $d_i$ with $r_i > 4 d_i$ and $d_i \to 0$ such that $d_i^\alpha m_{r_i} / r_i^\alpha > \eps$. Since $m_r$ is bounded above, $d_i^\alpha / r_i^\alpha$ cannot go to zero. But that would imply that $r_i \to 0$ and $d_i^\alpha m_{r_i} / r_i^\alpha < m_{r_i} / 4^\alpha \to 0$, which gives a contradiction.
\end{proof}

\begin{remark}
If we only considered H\"older modulus of continuity, the proof would not be much shorter but it would be more direct. It is nevertheless interesting that it can be made to work for arbitrary modulus of continuity.
\end{remark}

\begin{proof}[Proof of Theorem \ref{t:boundarycontinuity}]
We apply Lemma \ref{l:bcont} to both $u$ and $-u$ to obtain a modulus of continuity that applies from any point on the boundary $\bdary B_1$ to any point in space $\R^n$. Then we use Lemma \ref{l:insidecont} to finish the proof. 
\end{proof}

\begin{remark}
The conditions $\Mp u \geq -C$ and $\Mm u \leq C$ are implied by $u$ being a solution of a linear equation like
\[ \int_{\R^n} \si u(x,y) (2-\sigma) \frac {a(x,y)}{|y|^{n+\sigma}} \dd y = f(x) \]
no continuity whatsoever is required for the coefficient $a(x,y)$, just that $\lambda \leq a(x,y) \leq \Lambda$. The function $f$ is only required to be bounded. The modulus of continuity is independent of $\sigma$ as long as $\sigma > \sigma_0 >0$, but the result was proved for constant order $\sigma$. We could obtain a corresponding result for a solution $u$ of an equation with variable order $\sigma(x)$ if instead of Theorem \ref{t:ca}, we used the estimates from \cite{S1}. The disadvantage is that those estimates would blow up as $\sigma \to 2$.
\end{remark}

\section{Approximation results.} \label{s:approximationresults}
The purpose of this section is to show that if two equations are very close to each other, then so are their solutions. The proof is by compactness: if a sequence of solutions and corresponding equations converges, the limiting function is the unique viscosity solution of the limiting equation. We stress that, thanks to the concept of viscosity solutions, it suffices that the convergence of the functionals be weak (i.e. against smooth test functions).

We obtain our desired result in great generality in Lemmas \ref{l:stability} and \ref{l:approximation}. To prove those lemmas we need some previous technical results that we develop in this section.

In the following lemma, we notice that the test functions used in the definition of viscosity solutions can be reduced to only the polynomials of degree two.

\begin{lemma} \label{l:onlypolynomials}
In Definition \ref{d:viscositysolutions} it is enough to consider test functions $\varphi$ which are quadratic polynomials and neighborhoods $N$ which are balls centered at the touching point.
\end{lemma}

\begin{proof}
Assume Definition \ref{d:viscositysolutions} holds every time $\varphi$ is a quadratic polynomial and $N$ is a ball centered at the touching point. We will show it also holds for any $C^2$ function $\varphi$. 

Let $\varphi$ be any $C^2$ test function from above in a neighborhood $N$ (similar reasoning would apply for test functions from below). Without loss of generality we assume $\varphi$ touches $u$ at the origin. Let $p(x)$ be the following quadratic polynomial.
\[ p(x) = \langle (D^2 \varphi(0)+\eps I) x, x \rangle + x \cdot D \varphi(0) + \varphi(0) \]

So, $p \geq \varphi \geq u$ in a neighborhood of the origin $B_r$. Let 
\begin{align*}
v_\varphi(x) &= \begin{cases}
               \varphi(x) & \text{if } x \in N \\
               u(x) & \text{if } x \notin N
               \end{cases} \\
v_p(x) &= \begin{cases}
               p(x) & \text{if } x \in B_r \\
               u(x) & \text{if } x \notin B_r.
               \end{cases}
\end{align*}
We choose $r$ small so that $v_p(x) \leq v_\varphi(x) + 2 \eps |x|^2 \chi_{B_r}$.

By assumption, $\D(v_p,0) \geq f(0)$. On the other hand
\begin{align*}
\D(v_\varphi,0) - \D(v_p,0) &\geq \MLm(v_\varphi-v_p , 0) \\
&\geq -\MLp((v_\varphi-v_p)^- , 0) \\
&\geq -C \MLp( 2\eps |x|^2 \chi_{B_r} , 0) \geq -C \eps
\end{align*}
where in the last inequality we used that $\norm{\MLp} \leq C$.

Se we have $\D(v_\varphi,0) \geq \D(v_p,0) - C \eps \geq f(0) - C \eps$ for arbitrary $\eps>0$. Thus  $\D(v_\varphi,0) \geq f(0)$.
\end{proof}

We want to show that if $\D_k (u_k,x) = f_k(x)$ and $\D_k \to \D$, $u_k \to u$ and $f_k \to f$ in some appropriate way, then $\D(u,x) = f(x)$. It is indeed true that the convergence $\D_k \to \D$ can be taken with respect to the norm defined in section \ref{s:preliminaries}, however we need a stronger result with respect to a weaker version of convergence that we define below.

\begin{defn} \label{d:weakconvergence}
We say that $\D_k \to \D$ weakly in $\Omega$ if for every $x_0 \in \Omega$ and for every function $v$ of the form
\[
v(x) = \begin{cases}
       p(x) & \text{if } |x-x_0| \leq \rho \\
       u(x) & \text{if } |x-x_0| \geq \rho
       \end{cases}
\] where $p$ is a polynomial of degree two and $u \in L^1(\R^n,\omega)$, we have $\D_k(v,x) \to \D(v,x)$ uniformly in $B_{\rho/2} (x_0)$.
\end{defn}

The following lemma is an improvement of Lemma 4.5 from \cite{CS}.

\begin{lemma} \label{l:stability}
Let $\D_k$ be a sequence of uniformly elliptic operators with respect to some class $\LI$. Let us assume that Assumption \ref{a:kernelK} holds. Let $u_k$ be a sequence of lower semicontinuous functions in $\Omega$ such that
\begin{itemize}
\item $\D_k (u_k,x) \leq f_k(x)$ in $\Omega$.
\item $u_k \to u$ in the $\Gamma$ sense in $\Omega$.
\item $u_k \to u$ in $L^1(\R^n,\omega)$.
\item $f_k \to f$ locally uniformly in $\Omega$.
\item $\D_k \to \D$ weakly in $\Omega$ (with respect to $\omega$)
\item $|u_k(x)| \leq C$ for every $x \in \Omega$. 
\end{itemize}
Then $\D u \leq f$ in $\Omega$.
\end{lemma}

\begin{proof}
Let $p$ be a polynomial of degree $2$ touching $u$ from below at a point $x$ in a neighborhood $B_r(x)$.

Since $u_k$ $\Gamma$-converges to $u$ in $\Omega$, for large $k$, we can find $x_k$ and $d_k$ such that $p+d_k$ touches $u_k$ at $x_k$. Moreover $x_k \to x$ and $d_k \to 0$ as $k \to + \infty$.

Since $\D_k(u_k,x) \leq f_k(x)$, if we let
\[ v_k = \begin{cases}
          p + d_k & \text{in } B_r(x) \\
          u_k & \text{in } \R^n \setminus B_r(x) \ ,
         \end{cases} \]
we have $\D_k (v_k ,x_k) \leq f_k(x_k)$. Clearly $v = \lim_{k \to +\infty} v_k$.

Let $z \in B_{r/4}(x)$. We have
\begin{align*}
|\D_k (v_k ,z) &- \D(v,z)| \leq \abs{\D_k (v_k,z) - \D_k (v,z)} + \abs{\D_k (v,z) - \D (v,z)}\\
&\leq \max\left( |\MLp (v_k-v)(z)|, |\MLp (v-v_k)(z)| \right) + \abs{\D_k (v,z) - \D (v,z)} \\
&\leq \sup_{L \in \LI} \abs{ L (v_k - v)(z) } + \abs{\D_k(v,z) - \D(v,z)}\\ 
&\leq \left( \int_{\R^n \setminus B_{r/2}} |\si(v_k-v)(z,y)| K(y) \dd y \right) + \abs{\D_k(v,z) - \D(v,z)}\\
&\leq C \left( \int_{\R^n \setminus B_{r/2}} 2|v_k(y)-v(y)| \sup_{z \in B_{r/4}} \omega(y+z) \dd y \right) + C |v_k(z)-v(z)| +\abs{\D_k(v,z) - \D(v,z)} \\
&\leq C \norm{v_k-v}_{L^1(\R^n,\omega)} + C |v_k(z)-v(z)|  +\abs{\D_k(v,z) - \D(v,z)}
\end{align*}
Where we used Assumption \ref{a:kernelK} and \eqref{e:omegadisks}.

The term $\abs{\D_k(v,z) - \D(v,z)}$ goes to zero uniformly for $z \in B_{r/4}(x)$ since we assumed $\D_k \to \D$ weakly. The sequence $\norm{v_k-v}_{L^1(\omega)} \to 0$ and $u_k \to u$ uniformly in $B_r(x)$, therefore the right hand side in the above inequality goes to zero uniformly for $z \in B_{r/4}$. We obtain $\D_k(v_k,x) \to \D(v,x)$ uniformly in $B_{r/4}(x)$. 

We have that $\D v$ is continuous in $B_r(x)$. We now compute
\[ |\D_k(v_k,x_k) - \D (v,x)| \leq |\D_k (v_k,x_k) - \D (v,x_k)| + |\D (v,x_k) - \D (v,x)| \to 0 \ . \]

So $\D_k (v_k,x_k)$ converges to $\D(v,x)$, as $k \to +\infty$. Since $x_k \to x$ and $f_k \to f$ locally uniformly, we also have $f_k(x_k) \to f(x)$, which finally implies $\D(v,x) \leq f(x)$.
\end{proof}

\begin{remark}
Note that the weight $\omega$ is important. Consider the following example. Let $\D_k = -(-\lap)^{1-1/k}$ and $\D = \lap$. Then $\D$ is a local operator, and we may think that we can take $\omega=0$ and overlook the $L^1(\R^n,\omega)$ convergence. However we can consider the sequence of functions
\[ u_k(x) = \chi_{B_1}(x) (|x|^2-1) - M_k \chi_{\R^n \setminus B_{2k}} \ . \]
If we choose $M_k$ large enough, then we can make $\D_k u_k \leq 0$ in $B_1$ for every $k$, but $\D f = \lap f = 2n > 0$ in $B_1$.

The point is that in this example there is no single weight $\omega$ for which at the same time $u_k \to u$ in $L^1(\R^n,\omega)$ and $\norm{\D_n - \D} \to 0$ with respect to $\omega$.
\end{remark}

\begin{lemma} \label{l:onetestsubseq}
Let $v$ be a function
\[ v(x) = \begin{cases}
               p(x) & \text{if } x \in B_r \\
               u(x) & \text{if } x \notin B_r.
          \end{cases} \]
and let $\D_k$ be a sequence of uniformly elliptic operators respect to some class $\LI$ satisfying Assumption \ref{a:kernelK}. There is a subsequence $\D_{k_j}$ such that 
$f_{k_j}(x) := \D_{k_j}(v,x)$ converges uniformly in $B_{r/2}$.
\end{lemma}

\begin{proof}
All we need to do is find a uniform modulus of continuity for $f_k$ in $B_{r/2}$ so that the Lemma follows by Arzela-Ascoli.

Recall the notation $\tau_z u(x) = u(x+z)$. Given $x,y \in B_{r/2}$ with $|x-y| < r/8$, we have
\begin{align*}
f_k(x) - f_k(y) &= \D_k(v,x) - \D_k(v,y) \\
&\leq \MLp(v - \tau_{y-x} v, x) \\
\intertext{Note that $v - \tau_{y-x} v$ is just linear function in $B_{r/4}(x)$}
&\leq \int_{\R^n \setminus B_{r/4}} \si(v - \tau_{y-x} v,x,z) K(z) \dd z \\
\intertext{where $K$ is the kernel of Assumption \ref{a:kernelK}, thus we have}
&\leq C \int_{\R^n \setminus B_{r/4}} (v(x+z) + v(x-z) - 2v(x) -  v(y+z) - v(y-z) + 2v(y)) \omega(z) \dd z \\
&\leq C ( |v(x)-v(y)| + \norm{\tau_x v - \tau_y v}_{L^1(\R^n,\omega)} ) \\
&\leq C \left( \sup_{|\bar x-\bar y| \leq |x-y|, \ \bar x,\bar y \in B_{r/2}} |v(\bar x)-v(\bar y)| + \int_{\R^n} |\tau_{(\bar x - \bar y)} v - v| (\sup_{B_r(z)} \omega) \dd z \right) \leq c(|x-y|)
\end{align*}
Where we use \eqref{e:omegadisks} and $c(\rho)$ is defined as
\[ c(\rho) := C \left( \sup_{|x-y| \leq \rho, \ x,y \in B_{r/2}} |v(x)-v(y)| + \int_{\R^n} |\tau_{(x-y)} v - v| \omega(z) \dd z \right).
\]

Clearly, $c(\rho)$ is a modulus of continuity that depends on $v$ but not on $\D_k$. So the functions $f_k$ have a uniform modulus of continuity and there is a subsequence that converges uniformly by Arzela-Ascoli.
\end{proof}

\begin{thm} \label{l:subseqconvergesweakly}
Let $\D_k$ be a sequence of operators uniformly elliptic with respect to some class $\LI$ satisfying Assumptions \ref{a:boundedness} and \ref{a:kernelK}. There is a subsequence $\D_{k_j}$ that converges weakly.
\end{thm}

\begin{proof}
In order to make this proof, first we will construct a dense subset of the test functions $v$ of the form
\begin{equation} \label{e:v}
 v(x) = \begin{cases}
               p(x) & \text{if } x \in B_r \\
               u(x) & \text{if } x \notin B_r.
          \end{cases}
\end{equation}

This dense subset is constructed as follows. We know that $L^1(\R^n,\omega)$ is separable, so it has a countable dense subset $\{u_i\}$. Quadratic polynomials are a finite dimensional space which also has a countable dense subset $\{p_i\}$. Now, for every positive integer $k$, we construct the functions
\[ v_{k,i_1,i_2}(x) = \begin{cases}
               p_{i_1}(x) & \text{if } x \in B_{2^{-k}} \\
               u_{i_2}(x) & \text{if } x \notin B_{2^{-k}}.
          \end{cases} \]
So, given any function $v$ as in \eqref{e:v} and any $\eps>0$ we first choose $k$ such that $2^{-k} < r < 2^{-k+1}$, then we choose $i_1$ and $i_2$ such that $\norm{u_{i_2} - u}_{L^1(\R^n,\omega)} < \eps$, $|D^2 p_{i_1} - D^2 p| < \eps$, $|Dp_{i_1} - Dp| < \eps$ in $B_{2^{-k}}$, and $|p_{i_1} - p| < \eps$ in $B_{2^{-k}}$.

Since the set $\{v_{k,i_1,i_2}\}$ is countable, we can arrange it in a sequence $v_i$ of the form
\[
v_i(x)  = \begin{cases}
               p_{i}(x) & \text{if } x \in B_{r_i} \\
               u_{i}(x) & \text{if } x \notin B_{r_i}.
          \end{cases}
\]
such that for every $v$ as in \eqref{e:v}, there is a $v_i$ such that
\begin{equation} \label{e:vclose}
\begin{aligned}
\norm{v-v_i}_{L^1(\R^n,\omega)} &< \eps \\ 
|v-v_i| &< \eps && \text{in } B_{r/2} \\ 
|D v-D v_i| &< \eps && \text{in } B_{r/2} \\ 
|D^2 v-D^2 v_i| &< \eps && \text{in } B_{r/2}.
\end{aligned}
\end{equation}

By Lemma \ref{l:onetestsubseq}, for each $v_i$, we can take a subsequence $\D_{k_j}$ such that $\D_{k_j}(v_i)$ converges uniformly in $B_{r_i/2}$.

By a classical diagonal argument, there is a subsequence $\D_{k_j}$ such that for every $v_i$, $\D_{k_j}(v_i)$ converges uniformly in $B_{r_i/2}$. We call this limit $\D_\infty(v_i,x)$.

If $v$ is any test function, there is an $i$ such that $v$ is very close to $v_i$ in the sense of \eqref{e:vclose}. For any $x \in B_{r/2}$, we have
\begin{align*}
\D_k(v,x) - \D_k(v_i,x) &\leq \MLp(v - v_i, x) \\
&\leq \MLp((v - v_i) \chi_{B_{r/2}}, x) + \MLp((v - v_i) (1-\chi_{B_{r/2}}), x)\\
\intertext{Note that $v - v_i$ is a quadratic function in $B_{r/2}(x)$}
&\leq \eps \norm{\D_k} + \int_{\R^n \setminus B_{r/2}} \si(v - v_i,x,z) K(z) \dd z \\
\intertext{where $K$ is the kernel of Assumption \ref{a:kernelK}, thus we have}
&\leq \eps \norm{\D_k} + C \norm{v-v_i}_{L^1(\R^n,\omega)} \\
&\leq C \eps
\end{align*}

So, by choosing $v_i$, we can make $\D_k(v,x) - \D_k(v_i,x)$ as small as we wish in $B_{r/2}$, uniformly in $k$. So, for $j$ large, $|\D_{k_j}(v,x) - \D_\infty(v_i,x)| < 2\eps$ and thus $\D_{k_j}(v,x)$ is a Cauchy sequence in $L^\infty(B_{r/2})$. We define $\D_\infty(v,x)$ to be the uniform limit of this sequence in $B_{r/2}$.

Therefore we have shown that $\D_k(v,x)$ converges uniformly to $\D_\infty(v,x)$ in $B_{r/2}$. To finish the proof, all we must show is that this operator $\D_\infty$ can be extended to a  uniformly elliptic operator for all test functions $\varphi$. The key is to note that for any two test functions $v_1$ and $v_2$ of the form \eqref{e:v} we have $\MLm(v_1-v_2,x) \leq \D_k(v_1-v_2,x) \leq \MLp(v_1-v_2,x)$, and so this inequality passes to the limit and becomes $\MLm(v_1-v_2,x) \leq \D_\infty(v_1-v_2,x) \leq \MLp(v_1-v_2,x)$. Now, approximating an arbitrary test function $\varphi$ as in the proof of Lemma \ref{l:onlypolynomials} we get that there is a unique way to extend $\D_\infty$ to all test functions $\varphi$ such that $\D_\infty$ is uniformly elliptic respect to $\LI$. 
\end{proof}



\begin{lemma} \label{l:approximation}
For some $\sigma \geq \sigma_0 > 1$ and $\alpha <\sigma_0 - 1$ we consider nonlocal operators $\D_0$, $\D_1$ and $\D_2$ uniformly elliptic respect to $\LI_0(\sigma)$. Assume also that the boundary problem
\begin{align*}
\D_0 u &= 0 \text{ in } B_1 \\
u &= g \text{ in } \R^n \setminus B_1
\end{align*}
does not have more than one solution $u$ for any $g$ continuous and $|g(x)| \leq M(|x|+1)^{1+\alpha}$.

Given $M>0$, a modulus of continuity $\rho$ and $\eps>0$, there is an $\eta>0$ (small) and a $R>0$ (large) so that if $u$, $v$, $\D_0$, $\D_1$ and $\D_2$ satisfy
\begin{align*}
\D_0 (v,x) &= 0  && \text{in } B_1, \\
\D_1 (u,x) &\geq -\eta && \text{in } B_1, \\
\D_2 (u,x) &\leq \eta && \text{in } B_1, \\
u &= v && \text{in } \R^n \setminus B_1, \\
\norm{\D_1-\D_0} &\leq \eta && \text{in } B_1, \\
\norm{\D_2-\D_0} &\leq \eta && \text{in } B_1, \\
|u(x) - u(y)| &\leq \rho(|x-y|) && \text{for every } x \in B_R \setminus B_1 \text{ and } y \in \R^n \setminus B_1, \\
|u(x)| &\leq M(|x|+1)^{1+\alpha}.
\end{align*}
then $|u-v| < \eps$ in $B_1$.
\end{lemma}

\begin{proof}
Assume the result was false. There there would be a sequence $R_k$, $\D_0^{(k)}$, $\D_1^{(k)}$, $\D_2^{(k)}$, $\eta_k$, $u_k$, $v_k$, $\D_k$, $f_k$ such that $R_k \to \infty$, $\eta_k \to 0$, all the assumptions of the lemma are valid but nonetheless $\sup |u_k-v_k| \geq \eps$ in $B_1$.

Since $\D_0^{(k)}$ is a sequence of uniformly elliptic operators, by Lemma \ref{l:subseqconvergesweakly} we can take a subsequence that converges weakly to some nonlocal operator $\D_0$ that is also elliptic respect to the same class $\LI$. Moreover, since $\norm{\D_0^{(k)} - \D_1^{(k)}} \to 0$ and $\norm{\D_0^{(k)}-\D_2^{(k)}} \to 0$, then $\D_1^{(k)}$ and $\D_0^{(k)}$ also converge weakly to $\D_0$.

Since both $u_k$ and $v_k$ have a modulus of continuity $\rho$ on $\bdary B_1$, then by Lemma \ref{l:insidecont} there is a modulus of continuity $\tilde \rho$ that extends to the full unit ball $\overline B_1$. So $u_k$ and $v_k$ have a uniform modulus of continuity on $B_{R_k}$ with $R_k \to \infty$. We can find a subsequence that converges uniformly on compact sets, thus converging almost everywhere in $R^n$. By dominated convergence theorem, $u_k$ and $v_k$ converge in $L^1(\R^n,\omega)$ to some function $g$.

Let $u$ and $v$ be the limits of $u_k$ and $v_k$ respectively. Since $u_k$ and $v_k$ converge to $u$ and $v$ uniformly in $B_1$ and $\sup |u_k-v_k| \geq \eps$ in $B_1$, then $u$ and $v$ must be different. We use Lemma \ref{l:stability} and we get that $u$ and $v$ solve the same equation $\D_0 (u,x) = \D_0 (v,x)=0$. But then $u=v$ which is a contradiction.
\end{proof}

\begin{remark}
We will apply Lemma \ref{l:approximation} to $\D_0$ translation invariant in which case the uniqueness assumption for the boundary problem was proved in \cite{CS}.
\end{remark}

We also present the following simplified version of Lemma \ref{l:approximation} which will not be used in this paper but it is interesting by itself. The difference with respect to Lemma \ref{l:approximation} is that in the following version we fix the boundary value $g$ and in exchange we do not require a modulus of continuity in $B_R \setminus B_1$ but only on $\bdary B_1$.

\begin{lemma} \label{l:approximation2}
For some $\sigma \geq \sigma_0 > 1$ and $\alpha <\sigma_0 - 1$ we consider nonlocal operators $\D_0$, $\D_1$ and $\D_2$ uniformly elliptic respect to $\LI_0(\sigma)$. Assume also that the boundary problem
\begin{align*}
\D_0 u &= 0 \text{ in } B_1 \\
u &= g \text{ in } \R^n \setminus B_1
\end{align*}
does not have more than one solution $u$ for some given boundary condition $g$.

Assume $g$ is continuous on $\partial B_1$. Given $M>0$ and $\eps>0$, there is an $\eta>0$ (small) so that if $u$, $v$, $\D_0$, $\D_1$ and $\D_2$ satisfy
\begin{align*}
\D_0 (v,x) &= 0  && \text{in } B_1, \\
\D_1 (u,x) &\geq -\eta && \text{in } B_1, \\
\D_2 (u,x) &\leq \eta && \text{in } B_1, \\
u &= v = g && \text{in } \R^n \setminus B_1, \\
\norm{\D_1-\D_0} &\leq \eta && \text{in } B_1, \\
\norm{\D_2-\D_0} &\leq \eta && \text{in } B_1,
\end{align*}
then $|u-v| < \eps$ in $B_1$ (in this case $\eta$ depends on $g$).
\end{lemma}

\begin{proof}
The proof follows the same line as the one of Lemma \ref{l:approximation}. We do it by contradiction assuming the result was false. There there would be a sequence $R_k$, $\D_0^{(k)}$, $\D_1^{(k)}$, $\D_2^{(k)}$, $\eta_k$, $u_k$, $v_k$, $\D_k$, $f_k$ such that $\eta_k \to 0$, all the assumptions of the lemma are valid but nonetheless $\sup |u_k-v_k| \geq \eps$ in $B_1$.

The functions $u_k$ and $v_k$ have a fixed value $g$ outside $B_1$. Since $g$ is continuous on $\partial B_1$, by Lemma \ref{l:insidecont}, $u_k$ and $v_k$ are equicontinuous in $\overline{B_1}$. So there is a subsequence that converges uniformly in $\overline B_1$.

We continue as in the proof of Lemma \ref{l:approximation}. We can take a subsequence such that $\D_1^{(k)}$ and $\D_0^{(k)}$ converge weakly to $\D_0$. Let $u$ and $v$ the limits of $u_k$ and $v_k$ respectively, so we know that $\sup |u_k-v_k| \geq \eps$ in $B_1$. But by Lemma \ref{l:stability}, $u$ and $v$ must solve the same equation $\D_0 (u,x) = \D_0 (v,x)=0$. But then $u=v$ which is a contradiction.
\end{proof}

\section{$C^{1,\alpha}$ regularity for variable coefficient equations}
\label{s:estimate}

The main purpose of this section is to prove a $C^{1,\alpha}$ estimate for nonlocal equations that are not necessarily translation invariant. We will give the result in its most general form, and in the next section we will provide a variety of applications.

It is important to stress the concept of scaling. Since our proofs will involve successive rescaling, we will need some kind of scale invariance. It is not necessary to require a particular equation to be scale invariant, but instead we will consider our equations within a whole class of equations that is scale invariant for which our regularity result up to the boundary is supposed to apply, and we will jump freely from equation to equation in our rescaling process. We say that the class $\LI$ has scale $\sigma$ if every time the integro-differential operator with kernel $K(y)$ is in $\LI$, then the one with kernel $\lambda^{n+\sigma} K(\lambda y)$ is also in $\LI$ for any $\lambda < 1$. For example the class $\LI_0$ defined in \eqref{e:uniformellipticity} has scale $\sigma$, whereas the class $\LI_\ast$ defined in \eqref{e:lic1a1}-\eqref{e:lic1a2} does not because of condition \eqref{e:lic1a2}.

It is a straight forward computation to check that if $\LI$ has scale $\sigma$ and $u$ solves an equation $\D(u,x)= f(x)$ that is elliptic with respect to $\LI$, then the function $w(x) = \mu u(\lambda x)$ solves a uniformly elliptic equation with respect to the same class $\LI$, $\D_{\mu, \lambda} (w,x) = \lambda^\sigma \mu f(\lambda x)$. For example, if $\D(u,x)$ corresponds to an integro-differential operator
\[ \D(u,x) = \int_{\R^n} (u(x+y) + u(x-y) - 2u(x)) K(y) \dd y \]
then $\D_{\mu, \lambda}$ corresponds to the integro-differential operator
\[ \D_{\mu, \lambda}(u,x) = \int_{\R^n} (u(x+y) + u(x-y) - 2u(x)) \lambda^{n+\sigma} K(\lambda y) \dd y .\]
Note that the coefficient $\mu$ does not have any effect on a linear operator.

We will call $\LI_1$ to the largest scale invariant class contained in the class $\LI_\ast$ defined in \eqref{e:lic1a1}-\eqref{e:lic1a2}. This is the class of integro-differential operators with kernels $K$ such that
\begin{align} 
(2-\sigma)\frac{\lambda}{|y|^{n+\sigma}} &\leq K(y) \leq (2-\sigma)\frac{\Lambda}{|y|^{n+\sigma}} \label{e:siclass1}\\ 
|\grad K(y)| &\leq C |y|^{-n-\sigma-1} \qquad \text{in } \R^n \setminus \{0\}. \label{e:siclass2}
\end{align}
From Theorem \ref{t:c1a}, an equation $\D_0(u,x)=0$ has interior $C^{1,\bar \alpha}$ estimates for some $\bar \alpha > 0$ if $\D_0$ is uniformly elliptic respect to the class $\LI_1$.

Our main theorem in this section states essentially that if an equation $\D^{(0)}(u,x)=0$ is uniformly elliptic respect to a scale invariant class with interior $C^{1,\bar \alpha}$ estimates and we have another equation $\D(u,x)=f(x)$ for which the operators $\D(-,x)$ stay close to $\D^{(0)}$, then this equation also has interior $C^{1,\alpha}$ estimates for any $\alpha < \min(\bar \alpha, \sigma-1)$. Since our proof involves successive rescaling, we need to measure \emph{closeness} at every scale, so we define a distance in between operators that takes scaling of order $\sigma$ into account.

\begin{defn} \label{d:scalednorm}
Given $\sigma \in (0,2)$ and an operator $\D$, we define the rescaled operator as above
\[ \D_{\mu, \lambda} (w,x) = \lambda^\sigma \mu \D (\mu^{-1} w(\lambda^{-1} -), \lambda x) \]

The norm of scale $\sigma$ is defined as
\[ \norm{\D^{(1)} - \D^{(2)}}_\sigma = \sup_{\lambda<1} \norm{\D^{(1)}_{1,\lambda} - \D^{(2)}_{1,\lambda}} \]

where $\norm{.}$ is the norm defined in Definition \ref{d:norm}.
\end{defn}

The purpose of the rescaled operator is to stress that if $u$ solves the equation $\D(u,x)=f(x)$ in $B_\lambda$, then the rescaled function $w(x) = \mu u(\lambda x)$ solves an equation of the same ellipticity type $\D_{\mu,\lambda} (w,x) = \lambda^\sigma f(\lambda x)$ in $B_1$.

The following is the main theorem of this work.

\begin{thm} \label{t:main}
Assume $\sigma > \sigma_0 > 1$.
Let $\D^{(0)}$ be a fixed translation invariant nonlocal operator in a class $\LI \subset \LI_0(\sigma)$ with scale $\sigma$ and interior $C^{1,\bar \alpha}$ estimates (for example $\LI = \LI_1$).

Let $\D^{(1)}$ and $\D^{(2)}$ be two nonlocal operators, elliptic with respect to $\LI_0(\sigma)$, and assume that \[\norm{\D^{(0)} - \D^{(j)}}_\sigma < \eta\] for some $\eta>0$ small enough and $j=1,2$.

Let $u$ be a bounded function that solves the equation
\begin{align*}
\D^{(1)}(u,x) &\geq f_1(x) \qquad \text{in } B_1 \\
\D^{(2)}(u,x) &\leq f_2(x) \qquad \text{in } B_1 \\
\end{align*}
for a couple of bounded functions $f_1$ and $f_2$.

Then $u \in C^{1,\alpha}(B_{1/2})$ for any $\alpha < \min(\bar \alpha,\sigma_0-1)$ and
\[ \norm{u}_{C^{1,\alpha}(B_{1/2})} \leq C (\norm{u}_{L^\infty(\R^n)}+\norm{f_1}_{L^\infty(B_1)}+\norm{f_2}_{L^\infty(B_1)}) \]
and the estimate depends only on $\sigma_0$, $\lambda$, $\Lambda$ and dimension, but not on $\sigma$.
\end{thm}

The proof of this theorem is based on an approximation argument inspired in \cite{C2} using Lemma \ref{l:approximation}.

\begin{proof}
By scaling the problem, we can assume without loss of generality that 
\begin{align*}
\norm{f_1}_{L^\infty(B_1)}&<\eta, \\
\norm{f_2}_{L^\infty(B_1)}&<\eta, \\
\norm{u}_{L^\infty(\R^n)} &\leq 1,
\end{align*}
 and $u$ solves the equation in some large ball $B_{2R}$.

We will show that there is a $\lambda>0$ and a sequence of linear functions
\[ l_k(x) = a_k + b_k \cdot x , \]
such that
\begin{align}
\sup_{B_{\lambda^k}} |u-l_k| &\leq \lambda^{k(1+\alpha)} \label{e:c1}\\
|a_{k+1} - a_k| &\leq \lambda^{k(1+\alpha)} \label{e:c2}\\
\lambda^k|b_{k+1} - b_k| &\leq C_2 \lambda^{k(1+\alpha)} \label{e:c3} 
\end{align}

It is a standard computation that these conditions imply that $u$ is $C^{1,\alpha}$ at the origin (see \cite{C2}).

Let $l_0 = 0$. We proceed by induction. Assume we have \eqref{e:c1}-\eqref{e:c2}-\eqref{e:c3} up to some value of $k$. We will now show it for $k+1$. Consider
\[ w_k(x) = \frac{ [u - l_k] (\lambda^k x) }{\lambda^{k(1+\alpha)}}, \]
then since $\LI$ has scale $\sigma$, $w_k$ solves an equation of the same ellipticity type
\[ \D_k^{(1)}(w_k,x) = \D^{(1)}_{\lambda^{-k(1+\alpha)}, \lambda^k} (w_k,x) \geq  \lambda^{k(\sigma-1-\alpha)} f_1(\lambda^k x) \qquad \text{(and } \leq \text{ for } \D_k^{(2)} \text{ and $f_2$ )}\]

Since $\sigma-1-\alpha \geq \sigma_0 - 1 - \alpha > 0$, the right hand side becomes smaller as $k$ increases. Moreover, \[\norm{\D^{(j)}_k - \D^{(0)}_k} \leq \norm{\D^{(j)} - \D^{(0)}}_\sigma < \eta.\]

By the inductive hypothesis, we have that $|w_k|\leq 1$ in $B_1$. Let $\alpha < \alpha_1 <\min(\bar \alpha,\sigma_0)$, we will show that we can construct the sequence $l_k$ so that we also have
\begin{equation*}
w_k(x) \leq |x|^{1+\alpha_1} \qquad \text{if } x \notin B_1
\end{equation*}

We choose $R$ from Lemma \ref{l:approximation}, so that we have that $w_k$ is H\"older continuous in $B_R \setminus B_1$ and we can apply Lemma \ref{l:approximation} to the function $h$ that solves
\begin{align*}
\D^{(0)}_k h &= 0 && \text{in } B_1 \\
h(x) &= w_k(x) &&  \text{in } \R^n \setminus B_1 
\end{align*}
then $|w_k-h| < \eps(\eta,R)$ in $B_1$.

Since $\D^{(0)}_k$ is translation invariant and elliptic with respect to $\LI$, it has interior $C^{1,\bar \alpha}$ estimates. Let $\bar l = \bar a + \bar b \cdot x$ be the linear part of $h$ at the origin. Since $|h| \leq 1 + \eps(\eta,R)$ in $B_1$, $|\bar a| \leq 1 + \eps(\eta,R)$. By the $C^{1,\bar \alpha}$ estimates, $\bar b \leq C_2$. We also have that $|h(x) - \bar l(x)| \leq C_3 |x|^{1+\bar \alpha}$ in $B_{1/2}$. So, we have the following estimates:
\begin{align}
|w_k(x) - \bar l(x)| &\leq \eps(\eta,R) + C_3 \lambda^{1+\bar \alpha} && \text{in } B_{1/2} \label{e:b1}\\ 
|w_k(x) - \bar l(x)| &\leq \eps(\eta,R) + 2 + C_2 && \text{in } B_1 \setminus B_{1/2} \label{e:b2}\\
|w_k(x) - \bar l(x)| &\leq \eps(\eta,R) + 1 + |x|^{1+\alpha_1} + C_2 |x| && \text{in }\R^n \setminus B_1 \label{e:b3}
\end{align}

We will choose $\lambda$ small to be determined below, and then $\eta$ and $R$ so that $\eps(\eta,R) \leq \lambda^{1+\bar \alpha}$. The question is how to pick $\lambda$ so that we can complete the inductive step. We have
\begin{align*}
l_{k+1}(x) &= l_k(x) + \lambda^{k(1+\alpha)} \bar l\left( \frac x {\lambda^k} \right), \\
w_{k+1} &= \frac{ (u - l_{k+1})(\lambda^k x) }{\lambda^{k(1+\alpha)}} = \frac{ (w_k - \bar l)(\lambda x)}{\lambda^{1+\alpha}}.
\end{align*}

Let us analyze the bounds \eqref{e:b1}, \eqref{e:b2} and \eqref{e:b3} in the rescaled setting. We get
\begin{align*}
|w_{k+1}(x)| &\leq \lambda^{\bar \alpha - \alpha} + C_3 \lambda^{\bar \alpha  - \alpha} && \text{in } B_{\lambda^{-1} / 2} \\ 
|w_{k+1}(x)| &\leq \lambda^{\bar \alpha - \alpha} + (2 + C_2) \lambda^{-(1+\alpha)} && \text{in } B_1 \setminus B_{\lambda^{-1}/2} \\
|w_{k+1}(x)| &\leq \lambda^{\bar \alpha - \alpha} + 1 + |\lambda|^{\alpha_1-\alpha} |x|^{1+\alpha_1} + C_2 |\lambda|^{-\alpha} |x| && \text{in }\R^n \setminus B_{\lambda^{-1}}
\end{align*}

We start by choosing $\lambda$ small enough so that $(1+C_3)\lambda^{\bar \alpha - \alpha} < 1$ and we have $|w_{k+1}| \leq 1$ in $B_1$. If we choose $\lambda$ small enough, we also have the following estimates in each of the three rings
\begin{align*}
|w_{k+1}(x)| &\leq \lambda^{\bar \alpha - \alpha} + C_3 \lambda^{\bar \alpha  - \alpha} \leq 1 \leq |x|^{1+\alpha_1} && \text{in } B_{\lambda^{-1} / 2} \setminus B_1 \\ 
|w_{k+1}(x)| &\leq \lambda^{\bar \alpha - \alpha} + (2 + C_2) \lambda^{-(1+\alpha)} \leq C \lambda^{\alpha_1-\alpha} \lambda^{-(1+\alpha_1)} \leq |x|^{1+\alpha_1} && \text{in } B_{\lambda^{-1}} \setminus B_{\lambda^{-1}/2} \\
|w_{k+1}(x)| &\leq \lambda^{\bar \alpha - \alpha} + 1 + |\lambda|^{\alpha_1-\alpha} |x|^{1+\alpha_1} + C_2 |\lambda|^{-\alpha} |x| 
\leq |x|^{1+\alpha_1}  && \text{in }\R^n \setminus B_{\lambda^{-1}}
\end{align*}
so that $|w_{k+1}| < |x|^{1+\alpha_1}$ outside $B_1$. Moreover we observe that
\begin{align*}
|a_{k+1} - a_k| \leq \lambda^{k(1+\alpha)} \\
\lambda^k |b_{k+1} - b_k| \leq C_2 \lambda^{k(1+\alpha)}
\end{align*}
where $C_2$ is the same as above and depends only on the $C^{1,\alpha}$ estimates. We also have $u(x) - l_{k+1}(x) = \lambda^{k(1+\alpha)} \left(w(\frac x {\lambda^k}) - \bar l(\frac x {\lambda^k}) \right)$, thus
\[ |u(x) - l_{k+1}(x)| \leq \lambda^{(k+1)(1+\alpha)} \qquad \text{in } B_{\lambda^{k+1}} . \]
This completes the inductive step and the proof.
\end{proof}

\section{Applications} \label{s:applications}
Since the main theorem in section \ref{s:estimate} was stated in its more general way which is somewhat abstract, we provide concrete applications in this section aimed at showing the strength of this result.

\subsection{Linear equations with variable coefficients}
We start by stating the result of the theorem in the simplest case when the equation is linear and close to an operator in $\LI_1$. This is the integral version of the classical Cordes-Nirenberg theorem.

\begin{thm} \label{t:linear}
Let $\sigma >1$ and let $u$ be a bounded function that solves the equation
\[\D(u,x) := \int_{\R^n} \si u(x,y) (2-\sigma)\frac{a(x,y)}{|y|^{n+\sigma}} \dd y  = f(x) \qquad \text{in } B_1\]
Assume that $f \in L^\infty(\R^n)$ and $|a(x,y) - a_0(y)| < \eta$ for every $x \in B_1$ and some small $\eta>0$. Assume $a_0$ is a bounded function such that $K(y) = (2-\sigma) a_0(y)/|y|^{n+\sigma}$ satisfies \eqref{e:siclass1} and \eqref{e:siclass2}.
Then for any $\alpha < \sigma-1$, $u \in C^{1,\alpha}(B_{1/2})$ if $\eta$ is small enough and
\[ \norm{u}_{C^{1,\alpha}(B_{1/2})} \leq C (\norm{u}_{L^\infty(\R^n)}+\norm{f}_{L^\infty(B_1)}). \]
\end{thm}

\begin{proof}
We apply Theorem \ref{t:main}. In this case $\D^{(0)}$ is given by
\[ \D^{(0)} (u,x) = \int_{\R^n} \si u(x,y) \frac{(2-\sigma)a_0(y)}{|y|^{n+\sigma}} \dd y \]

By assumption, this operator $\D^{(0)}$ belongs to $\LI_1$, which is a scale invariant class with interior $C^{1,\alpha}$ estimates. In this case it is easy to see that since the equation is linear and the coefficients do not depend on $x$, the derivatives of the solution solve the same equation and the solutions are actually $C^{2,\alpha}$. So in particular $\D^{(0)}$ has interior $C^{1,1}$ estimates.

The weight $\omega$ must be a weight that controls the tails of the kernel. In this case a natural choice is
\[ \omega(y) = \frac{1}{(1+|y|)^{n+\sigma}}. \]

Now we estimate $\norm{\D-\D^{(0)}}_\sigma$,
\[\norm{\D-\D^{(0)}}_\sigma = \sup_{\lambda<1} \norm{\D_{1,\lambda} - (\D^{(0)})_{1,\lambda}}. \]
We compute
\[ (\D_{1,\lambda} - (\D^{(0)})_{1,\lambda}) (u,x) = \int_{\R^n} \si u(x,y) \frac{(2-\sigma)(a(\lambda x,\lambda y)-a_0(\lambda y))}{|y|^{n+\sigma}} \dd y \]
with $|a(\lambda x,\lambda y)-1| < \eta$. From Definition \ref{d:norm}
\begin{equation}
\begin{split}
\norm{\D_{1,\lambda} - (\D^{(0)})_{1,\lambda}} &= \sup_{u \text{ as in Definition \ref{d:norm}}} \frac{\D_{1,\lambda} (u,x) - (\D^{(0)})_{1,\lambda} (u,x)}{1+M} \\
&\leq \sup_{u \text{ as in Definition \ref{d:norm}}} \frac{1}{1+M}\int_{\R^n} |\si u(x,y)| \frac{(2-\sigma)\eta}{|y|^{n+\sigma}} \dd y \label{e:s01}
\end{split}
\end{equation}

In Definition \ref{d:norm} we take functions $u$ such that $\norm{u}_{L^1(\R^n,\omega)}$ and $|u(y)-u(x)-(y-x)\cdot \grad u(x)| \leq M|x-y|^2$ for $y \in B_1(x)$. Note that these two conditions together imply $|u(x)| \leq CM$ for some universal constant $C$.

From $|u(y)-u(x)-(y-x)\cdot \grad u(x)| \leq M|x-y|^2$, we obtain
\begin{align*}
\int_{B_1} |\si u(x,y)| \frac{(2-\sigma)}{|y|^{n+\sigma}} \dd y &\leq \int_{B_1} M|y|^2 \frac{(2-\sigma)}{|y|^{n+\sigma}} \dd y \\
&\leq CM
\end{align*}

On the other hand, from $\norm{u}_{L^1(\R^n,\omega)}$ and $|u(x)| \leq CM$ we bound the tail of the integral
\begin{align*}
\int_{\R^n \setminus B_1} |\si u(x,y)| \frac{(2-\sigma)}{|y|^{n+\sigma}} \dd y &\leq \int_{\R^n \setminus B_1} |u(x+y)|+|u(x-y)|+|u(x)| \frac{(2-\sigma)}{|y|^{n+\sigma}} \dd y \\
&\leq CM
\end{align*}

Replacing in \eqref{e:s01}, $\norm{\D_{1,\lambda} - \D^{(0)}_{1,\lambda}} \leq C \eta$ for any $\lambda$, thus $\norm{\D-\D^{(0)}}_\sigma \leq C \eta$. So if we choose $\eta$ small enough we can apply Theorem \ref{t:main} and conclude that the equation $\D(u,x)=f$ has interior $C^{1,\alpha}$ estimates for any $\alpha < \sigma-1$.
\end{proof}

It is interesting to note in the theorem above that if we fix an $\alpha>0$, the constant $C$ in Theorem \ref{t:linear} does not blow up as $\sigma \to \infty$.

\subsection{Nonlinear equations with variable coefficients}

By combining Theorem \ref{t:main} with Theorem \ref{t:c1a} we obtain the following result
\begin{thm} \label{t:main2}
Let $\sigma \geq \sigma_0 >1$ and let $\LI_1(\sigma)$ be the class defined in \eqref{e:siclass1}--\eqref{e:siclass2}.

Assume $\D^{(0)}$ is translation invariant nonlocal operator uniformly elliptic respect to $\LI_1(\sigma)$ and $\D$ uniformly elliptic respect to $\LI_0$. Assume $\norm{\D(- ,x)- \D^{(0)}}_\sigma < \eta$ for every $x \in B_1$ and some small $\eta>0$. Let $f \in L^\infty(B_1)$ and $u$ be a bounded function that solves the equation
\[\D(u,x) = f(x) \qquad \text{in } B_1.\]
Then $u \in C^{1,\alpha}(B_{1/2})$ for some small $\alpha>0$ and
\[ [u]_{C^{1,\alpha}(B_{1/2})} \leq C (\norm{u}_{L^\infty(\R^n)}+\norm{f}_{L^\infty(B_1)}). \]
and the estimate depends only on $\sigma_0$, $\lambda$, $\Lambda$, $C_1$ and dimension, but not on $\sigma$.
\end{thm}

A concrete example of a nonlinear operator $\D$ such that Theorem \ref{t:main2} applies would be
\[ \D(u,x) = \inf_\alpha \sup_\beta \int_{\R^n} (u(x+y)+u(x-y)-2u(x)) \frac{(2-\sigma)(a_0(y) + a_{\alpha \beta}(x,y))}{|y|^{n+\sigma}} \dd y \]
such that we have for some small $\eta$,
\begin{align*}
 |a_{\alpha \beta}(x,y)| &< \eta \qquad \text{ for every } \alpha, \beta \\
 \lambda &\leq a_0(y) \leq \Lambda \\
 |\grad a_0(y)| &\leq C |y|^{-1}
\end{align*}

Of course, a regularity result that requires coefficients not to vary much with respect to $x$ implies a regularity estimate every time the coefficients are continuous in $x$. We state it in the following theorem.
\begin{thm}
Given $\sigma \geq \sigma_0 > 1$, let $\D(u,x)$ be given by
\[ \D(u,x) = \inf_\alpha \sup_\beta \int_{\R^n} \si u(x,y) \frac{(2-\sigma)a_{\alpha \beta}(x,y)}{|y|^{n+\sigma}} \dd y \]
such that for every $\alpha$, $\beta$ we have $\lambda < a_{\alpha \beta}(x,y) < \Lambda$ and $\grad_y a_{\alpha \beta}(x,y) \leq C_1/((2-\sigma)|y|)$ and also $|a_{\alpha \beta}(x_1,y) - a_{\alpha \beta}(x_2,y)|<c(|x_1-x_2|)$ for some uniform modulus of continuity $c$. Then solutions $u$ to the equation
\[ \D(u,x) = f \qquad \text{in } \Omega \]
are of class $C^{1,\alpha}(B_{1/2})$ for some small $\alpha>0$ and
\[ [u]_{C^{1,\alpha}(B_{1/2})} \leq C (\norm{u}_{L^\infty(\R^n)}+\norm{f}_{L^\infty(B_1)}). \]
and the estimate depends only on $\sigma_0$, $\lambda$, $\Lambda$, $C_1$, the modulus of continuity $c$ and dimension, but not on $\sigma$.
\end{thm}

\begin{proof}
Around each point $x_0 \in B_{1/2}$ we can find a ball $B_r(x_0)$ so that $|a_{\alpha \beta}(x,y) - a_{\alpha \beta}(x_0,y)|<c(r) \leq \eta$. This implies that $\norm{\D(-,x)-\D(-,x_0)}_\sigma < C\eta$ as in the proof of Theorem \ref{t:linear}, and we apply Theorem \ref{t:main2} with $\D^{(0)} = \D(-,x_0)$, scaled in $B_r(x)$.
\end{proof}

\subsection{Nonlinear equations with constant coefficients but non-differentiable kernels}
It is interesting to note that Theorem \ref{t:main} is useful to obtain new results even in the translation invariant case. In Theorem \ref{t:c1a}, it is required that every kernel must be differentiable away from the origin. That assumption can be relaxed in the following way. We can actually obtain a $C^{1,\alpha}$ estimate for nonlocal equations that are uniformly elliptic respect to the class $\LI$ given by operators with kernel $K$ such that
\begin{align*}
K(y) &= (2-\sigma) \frac{a_1(y) + a_2(y)}{|y|^{n+\sigma}} \\
\lambda &\leq a_1(y) \leq \Lambda \\
|a_2| &\leq \eta \\ 
|\grad a_1(y)| &\leq \frac{C_1}{|y|} \qquad \text{in } \R^n \setminus \{0\}
\end{align*}

\begin{thm}
Assume $\sigma \geq \sigma_0 > 1$. If $\eta$ is small enough (depending on $\lambda$, $\Lambda$, $C_1$ and dimension) then the nonlocal equation
\[ \D(u,x) = \inf_\alpha \sup_\beta \int_{\R^n} (u(x+y)+u(x-y)-2u(x)) K_{\alpha \beta}(y) \dd y  = f(x)\]
has interior $C^{1,\alpha}$ estimates if all kernels belong to the class $\LI$ above.
\end{thm}

\begin{proof}
Let $L \in \LI$ be an operator with kernel $K$, we can write it as $K=K^1+K^2$ where $K^1 =(2-\sigma) \frac{a_1(y)}{|y|^{n+\sigma}}$ and correspondingly $L = L^1 + L^2$. But then we see that $\norm{L-L^1}_\sigma < c \eta$. So we write
\[ \D^{(0)} (u,x) = \inf_\alpha \sup_\beta L^1_{\alpha \beta} u(x) \]
and we have $\norm{\D - \D^{(0)}}_\sigma < c \eta$ so that we can apply Theorem \ref{t:main}.
\end{proof}

This theorem applies for a class that is still smaller than $\LI_0$. It would be interesting to determine whether the class $\LI_0$ has interior $C^{1,\alpha}$ estimates or not. We leave that question open.

\subsection{Nonlinear equations close to the fractional Laplacian. Very regular solutions.}
We can obtain another new result in the translation invariant case from Theorem \ref{t:main}. In this case we look at non linear equations that are sufficiently close to the fractional Laplacian and we obtain $C^{2,\alpha}$ regularity. We provide a theorem on the regularity of fully nonlinear translation invariant nonlocal equations when the ellipticity constants are sufficiently close to each other. This is an improvement on Theorem \ref{t:c1a} when these conditions are satisfied.

\begin{thm} \label{t:c2aIflLareClose}
Assume $\sigma>\sigma_0>1$. There is a $\eta>0$ and $\rho_0>0$ so that if $\Lambda<1+\eta$, $\lambda > 1-\eta$, $\D$ is a nonlocal translation invariant uniformly elliptic operator with respect to $\LI_\ast$ (given in \eqref{e:lic1a1}-\eqref{e:lic1a2}) and $u$ is a continuous function in $\overline B_1$ such that $u \in L^1(\R^n,\omega)$, $\D u = 0$ in $B_1$,
then there is a universal (depends only on $n$ and $\sigma_0$) $\alpha>0$ such that $u \in C^{2+\alpha}(B_{1/2})$ and
\[ u_{C^{2+\alpha}(B_{1/2})} \leq C\left( \sup_{B_1} |u| + \norm{u}_{L^1(\R^n,\omega)} +|I0| \right) \]
for some constant $C>0$ (where by $I0$ we mean the value we obtain when we apply $\D$ to the function that is constant equal to zero). The constant $C$ depends on $\sigma_0$, $n$ and the constant in \eqref{e:lic1a2}.
\end{thm}
 
\begin{proof}
From Theorem \ref{t:c1a}, we know that $u \in C^{1,\alpha}(B_{3/4})$. In particular the function is differentiable. Let $w=u_e$ be a directional derivative. We proceed as in the proof of Theorem \ref{t:c1a} (Theorem 12.1 in \cite{CS}) and write $w = w_1 +w_2$ where $w_2$ vanishes in $B_{5/8}$ and $w_1$ solves
\begin{align*}
\Mp_{\LI_\ast} w_1 &\geq -C && \text{in } B_{5/8} \\ 
\Mm_{\LI_\ast} w_1 &\leq C && \text{in } B_{5/8} 
\end{align*}
The difference with the proof of Theorem 12.1 in \cite{CS} is that now we apply Theorem \ref{t:main} instead of Theorem \ref{t:ca}. Since $\lambda > 1-\eta$ and $\Lambda<1+\eta$, then $\norm{\Mp_{\LI_\ast} w_1 + (-\lap)^\sigma} < c\eta$ and $\norm{\Mm_{\LI_\ast} w_1 + (-\lap)^\sigma} < c\eta$. So Theorem \ref{t:main} tells us that $w = u_e$ is $C^{1,\alpha}$ in $B_{1/2}$, and thus $u \in C^{2,\alpha}(B_{1/2})$. 
\end{proof}

Theorem \ref{t:main} is also useful to reduce the equation to its leading order part. A simple example would be the following.

\begin{thm} \label{t:linearpluslowerorder}
Let $\sigma \in (1,2)$ and $r>0$. Let $u$ be a bounded function that solves the equation
\[\D(u,x) =  -(-\lap)^{\sigma/2} u(x) + \inf_\alpha \sup_\beta \int_{\R^n} \si u(x,y) (2-\sigma) \frac{|y|^r a_{\alpha \beta}(y)}{|y|^{n+\sigma}} \dd y = 0 \qquad \text{in } B_1\]
Assume that $f \in L^\infty(\R^n)$, $r>0$, $|a(y)| < C$, and $|\grad a(y)| \leq C/|y|$. Then $u \in C^{2,\alpha}(B_{1/2})$ for some $\alpha >0$ and
\[ \norm{u}_{C^{2,\alpha}(B_{1/2})} \leq C \norm{u}_{L^\infty(\R^n)}. \]
\end{thm}

\begin{proof}
First of all let us notice that we can place the fractional Laplacian inside the $\inf \sup$ and rewrite the equation as
\[\D(u,x) =  \inf_\alpha \sup_\beta \int_{\R^n} \si u(x,y) (2-\sigma) \frac{c_\sigma + |y|^r a_{\alpha \beta}(y)}{|y|^{n+\sigma}} \dd y = 0 \qquad \text{in } B_1\]
where $c_\sigma$ is the normalization constant of the fractional Laplacian which remains bounded below and above for $\sigma \in (1,2)$.

When we make a rescaling of the equation of order $\sigma$ we obtain
\[ \D_{1,\lambda} = \inf_\alpha \sup_\beta \int_{\R^n} \si u(x,y) (2-\sigma)\frac{1 + \lambda^r |y|^r a_{\alpha \beta}(x,\lambda y)}{|y|^{n+\sigma}} \dd y.\]
Then, for $\lambda$ small enough, we will have $\norm{\D_{1,\lambda}+(-\lap)^{\sigma/2}}_\sigma < \eta$ and we proceed as in the proof of Theorem \ref{t:c2aIflLareClose}.
\end{proof}

\section*{Appendix}

In this appendix, we prove Lemma \ref{l:barrier}. We can give an intuitive proof avoiding lengthy computations by using Lemma \ref{l:stability}. We notice that we did not use Lemma \ref{l:barrier} to prove Lemma \ref{l:stability}.

\begin{proof}[Proof of Lemma \ref{l:barrier}]
For some $r>0$, we will show that if we consider $u_\alpha(x) = ((|x|-1)^+)^\alpha$ and $x_0=(1+r)e_1$, then $M^+ u(x_0) \leq 0$ if $\alpha$ is small enough. By rotation this implies the result for any $x$ such that $|x|=1+r$. If $1<|x|<1+r$, we can still get the result by a scaling centered at $e_1$ (i.e. by considering $\tilde u(x) = ((|x+\lambda e_1|-(1+\lambda))^+)^\alpha$ for $\lambda$ such that $|x_0|-1 = r (1+\lambda)$), since that scaling would only decrease the values of the function elsewhere. This observation implies also that if $M^+ u(x_0) \geq 0$, then also $M^+ u(x) \geq 0$ for all $x$ such that $|x| \leq |x_0|$.

Let $M$ be a large number and let $v_j$ be
\[ v_\alpha(x) = \max \left(\frac{u_\alpha(x) - 1}{\alpha},-M \right). \]
The key observation is the following
\[ \lim_{\alpha \to 0} v_\alpha(x) = \max(\log(|x|-1),-M) = l_M(x)\]
and this convergence holds uniformly on compact sets. 

As $\sigma \to 2$, $M^+_\sigma$ converges to a uniformly elliptic second order differential operator that is comparable to the maximal Pucci operator with ellipticity constants depending on $\lambda$, $\Lambda$ and dimension, that we call $M^+_2$. It is a simple computation that $M^+(D^2 \log(|x|-1)) \to -\infty$ as $|x| \to 1$ (where $M^+(D^2u)$ stands for the standard Pucci operator). So if we pick $r$ small enough and $M$ larger than $-\log r$, then there will be a $\sigma_1$ such that $M^+_\sigma l_M < -1$ if $\sigma > \sigma_1$. Moreover, by choosing $M$ large enough, we can also make sure that $M^+_\sigma l_M < -1$ for $\sigma \leq \sigma_1$. This is how we pick $r$ and $M$.

Assume that the result of the lemma was not true. So there would be a sequence $\alpha_j \to 0$ and $\sigma_j \in [\sigma_0,2]$ such that $M^+_{\sigma_j} u_{\alpha_j} (x_0)\geq 0$, and thus $M^+_{\sigma_j} v_{\alpha_j} \geq 0$ in the annulus $B_{2|x_0|} \setminus B_{|x_0|}$. We extract a subsequence $\sigma_{j_k}$ that converges to some number $\tilde \sigma \in [\sigma_0,2]$. We apply Lemma \ref{l:stability} and we obtain that $M^+_{\tilde \sigma} l_M (x_0) \geq 0$. But this is a contradiction since $M^+_{\tilde \sigma} l_M (x_0) \leq -1$.
\end{proof}

\bibliographystyle{plain}   
\bibliography{nonl}             
\index{Bibliography@\emph{Bibliography}}%
\end{document}